\documentclass[journal,twoside,web]{ieeecolor}
\usepackage{generic}
\usepackage{cite}
\usepackage{amsmath,amssymb,amsfonts}
\usepackage{algorithmic}
\usepackage{graphicx}
\usepackage{algorithm,algorithmic}
\usepackage[hidelinks]{hyperref}
\usepackage[utf8]{inputenc} 

\usepackage{csquotes}  
\newcommand\q{\enquote}

\usepackage{textcomp}

\usepackage[english]{babel} 
\usepackage{mathtools} 
\graphicspath{{images/}} 
\usepackage{Macros}
\definecolor{revisiongreen}{rgb}{0,0.5,0}
\definecolor{revisionblue}{rgb}{0.5,0,0.7}
\newcommand{\revblue}[1]{{\color{blue}#1}} 
\renewcommand{\revblue}[1]{{#1}} 
\newtheorem{lemma}{Lemma}[section]

\newtheorem{proposition}[lemma]{Proposition}
\newtheorem{theorem}[lemma]{Theorem}

\newtheorem{corollary}[lemma]{Corollary}
\newtheorem{definition}[lemma]{Definition}
\newtheorem{example}[lemma]{Example}

\newtheorem{remark}[lemma]{Remark}

\renewcommand\theenumi{\arabic{enumi})}

\allowdisplaybreaks[4]  

\usepackage{tikz}
\usepackage{pgfplots}
\pgfplotsset{compat=1.18}
\usetikzlibrary{positioning}
\usetikzlibrary{cd}
\usetikzlibrary{shapes,fit} 

\usetikzlibrary{decorations.markings}
\tikzset{degil/.style={
		decoration={markings,
			mark= at position 0.5 with {
				\node[transform shape] (tempnode) {$\backslash$};
			}
		},
		postaction={decorate}
	}
}
\definecolor{uniwueblue}{HTML}{1F5394}
\pgfplotsset{soldot/.style={color=uniwueblue,only marks,mark=*, mark options={scale=.3mm}}}
\renewcommand\theenumi{(\roman{enumi})}

\begin{document}
\title{Lyapunov characterization of boundedness of reachability sets for infinite-dimensional systems}

\author{Patrick Bachmann$^{1}$,~\IEEEmembership{Student Member,~IEEE,} and 
Andrii Mironchenko$^{2}$,~\IEEEmembership{Senior Member,~IEEE} 
\thanks{*P. Bachmann (corresponding author) and A. Mironchenko have been supported by the German Research Foundation (DFG), grants Nr. 552243517 (MI 1886/5-1) and 540580186 (MI 1886/3-1), respectively.} 
\thanks{$^{1}$P. Bachmann is with the Institute of Mathematics, University of W{\"u}rzburg, Germany {\tt\small patrick.bachmann@uni-wuerzburg.de}}%
\thanks{$^{2}$A. Mironchenko is with the Department of Mathematics, University of Bayreuth, Germany,\\ \hfill {\tt\small andrii.mironchenko@uni-bayreuth.de}}%
}

\maketitle
\thispagestyle{empty}
\pagestyle{empty}

\begin{abstract}
    We prove a converse Lyapunov theorem for boundedness of reachability sets for a general class of control systems whose flow is Lipschitz continuous on compact intervals with respect to trajectory-dominated inputs. We show that this condition is satisfied by \revblue{broad classes of semilinear evolution equations}. 
    For ordinary differential equations, as a consequence of our results, we obtain a converse Lyapunov theorem for forward completeness, without a priori restrictions on the magnitude of inputs. \revblue{As an example, we rederive Wintner's classical result using our Lyapunov approach.}
\end{abstract}


\section{Introduction}\label{sec:introduction}

A control system is called forward complete if for any initial condition $x$ and any input $u$, the corresponding trajectory $\phi(\ph,x,u)$ is well-defined on the whole nonnegative time axis. If additionally, for any magnitude $R>0$ and any time $\tau>0$
\[
\sup_{\|x\| \leq R,\ \|u\| \leq R,\ t\in[0,\tau]} \|\phi(t,x,u)\| <+\infty,
\] 
then a control system is said to have bounded (finite-time) reachability sets.

\revblue{Although analysis of reachability and reachability sets is a classical topic in control theory, the importance of \emph{boundedness of reachability sets (BRS)}  in its own right was first recognized in \cite{LSW96}, where it played an important role in the derivation of converse Lyapunov theorems for uniform global asymptotic stability of systems with disturbances.}

BRS is essential for establishing regularity properties of the flow maps for (in)finite-dimensional nonlinear systems \cite[Theorem 1.40]{Mir23},  \cite[Section 3.5]{Mir24}. 
 Criteria for input-to-state stability in terms of the uniform limit property, local stability, and boundedness of reachability sets have been shown for general nonlinear control systems in \cite{MiW18b}. These characterizations, in turn, paved the way for the development of non-coercive Lyapunov methods \cite{MiP20}.

Robust forward completeness (RFC) is closely related to the boundedness of reachability sets, and was important for the characterization of
uniform global asymptotic stability for general infinite-dimensional systems in terms of uniform weak attractivity, local stability, and RFC in \cite{MiW19a}, and for the characterization of global asymptotic stability for retarded systems \cite{KPC22}.%
\footnote{\revblue{We use for BRS and RFC properties the terminology from \cite{Mir23e}. Originally, \cite[Definition 2.2]{Kar04} introduced RFC for systems with two types of inputs, with RFC features in one input and BRS-like behavior in the second. \revblue{Another variant is used in \cite{MiW19a}.}}}

Sufficient conditions for forward completeness for ordinary differential equations (ODEs) and other classes of control systems are a classic\revblue{al} subject \cite{Win45, Had72, Jus67, Tan92, Iwa83}. In particular,
if $f:\R^n\to\R^n$ is globally Lipschitz or \revblue{locally Lipschitz and} linearly bounded, then the solutions of
\[
\dot{x} = f(x)
\]
exist globally, and the reachability sets are bounded. More general sufficient conditions are given by Wintner's theorem \cite{Win45}.

\revblue{Necessary conditions for forward completeness have been investigated much less.}
Necessary and sufficient conditions of Lyapunov type for forward completeness of ODEs without inputs have been proposed in \cite{KaS67}. However, \revblue{the Lyapunov functions constructed in \cite{KaS67} are time-varying even for time-invariant ODEs.}

In \cite{LSW96}, for systems 
\begin{align}
\label{eq:intro-ODE}
    \dot{x} = f(x,u)
\end{align}
with Lipschitz continuous $f$, it was shown that forward completeness is equivalent to the boundedness of reachability sets.
In \cite{AnS99}, the authors have shown that for ODEs with uniformly bounded inputs ($\esssup_{t\geq 0}|u(\revblue{t})| \leq R$ for some \emph{fixed} $R>0$), forward completeness is equivalent to the existence of a Lyapunov-like function that increases at most exponentially.

For distributed parameter systems, the situation is more complex. Linear forward complete infinite-dimensional systems necessarily have bounded reachability sets \cite[Proposition 2.5]{Wei89b}. However, nonlinear forward complete infinite-dimensional systems with \revblue{right-hand sides that are Lipschitz continuous on bounded balls,} do not necessarily have bounded reachability sets, as demonstrated in \cite[Example 2]{MiW18b} for nonlinear infinite networks. 
\revblue{A stronger version of this example was given in \cite[Example 3.19]{Bac26}, where an infinite network with Lipschitz continuous nonlinearities was constructed, which is forward complete, locally exponentially stable, with all solutions approaching any given $\eps$-neighborhood in finite time, but which still does not have bounded reachability sets. 
}

The subtle relation between forward completeness and boundedness of reachability sets is best reflected by time-delay systems. For delay systems defined over the canonical state space of continuous functions, forward completeness is weaker than BRS \cite{MaH24}, which has profound implications such as the loss of the uniformity of the limit property.  
However, forward completeness is equivalent to BRS for delay systems with pointwise delays and over the state space of $\LL^\infty$ functions \cite{BCM24}. 
\revblue{Equivalence also holds if the state space is chosen to be a Sobolev space \cite{KPC22}, and for a certain class of delay systems with a combination of discrete and distributed delays \cite{HaM26}}.

These relations indicate that the BRS property (establishing uniform bounds for solutions on finite time intervals) is a bridge between pure well-posedness theory (that studies existence and uniqueness, but does not care much about the bounds for solutions) and stability theory (which is interested in establishing certain bounds for solutions for all nonnegative times, as well as their convergence).

\revblue{For time-delay systems, sufficient conditions for RFC in terms of Lyapunov functions are discussed in the survey paper \cite[Sec. 2.7]{CKPW23}.} In \cite{Mir23e}, a necessary and sufficient Lyapunov criterion for RFC was provided. Furthermore, \cite{Mir23e} introduces the concept of a BRS Lyapunov function and shows that for a system satisfying the boundedness-implies-continuation (BIC) property, existence of a BRS Lyapunov function implies BRS. The converse direction, however, remained open, as the strategy used in \cite{Mir23e} cannot be extended to the more complex BRS case without methodological changes.

\textbf{Contribution:} \revblue{In this work, we consider a general class of control systems including ODEs, evolution partial differential equations, switched and time-delay systems. Under certain regularity assumptions, we show that \emph{BRS is equivalent to the existence of a BRS Lyapunov function, which is Lipschitz continuous on bounded sets.} 
Establishing Lipschitz continuity of Lyapunov functions is the key technical difficulty in the proof. To obtain it, we require that the
flow is \emph{Lipschitz continuous on compact intervals with respect to trajectory-dominated inputs}, and that growth and regularity bounds of the flow coincide. This novel concept can be interpreted as Lipschitz continuity of solution multifunctions.

We show that for semilinear evolution equations with nonlinearities that are Lipschitz continuous on bounded sets and, as a subclass, ODEs with right-hand sides that are Lipschitz continuous on bounded sets, BRS implies Lipschitz continuity of the flow on compact intervals with respect to trajectory-dominated inputs, as well as the required relation between growth and regularity bounds.

We start our analysis by showing that BRS implies a new property that we call \emph{robust forward completeness (RFC) with respect to trajectory-dominated inputs} (i.e., inputs bounded by a $\KK_\infty$-function of the norm of the trajectory). This property has a quantitative representation that allows us to use ideas from the proof of the classical RFC converse Lyapunov theorem developed in \cite{Mir23e} to prove our BRS converse Lyapunov theorem.

We illustrate our results by deriving Wintner's sufficient conditions for forward completeness of ODE systems using our Lyapunov technique.}

\revblue{\textbf{Proof technique:}}
\revblue{The classical technique for proving converse Lyapunov results} for systems with inputs (such as \eqref{eq:intro-ODE}, but also of more general ones) was proposed in \cite{SoW95}, where a converse Lyapunov theorem for input-to-state stability (ISS) of systems with Lipschitz nonlinearities \eqref{eq:intro-ODE} has been proved. The proof in \cite{SoW95} consists of two steps: \revblue{F}irst, one proves the equivalence of ISS and robust stability of an auxiliary system with \emph{noisy state feedback}, where the inputs belong 
to a closed bounded ball, and then one invokes a converse Lyapunov result \cite{LSW96} for this auxiliary system. In \cite{MiW17c}, this strategy was used to prove the converse Lyapunov result for ISS of infinite-dimensional systems, using the converse Lyapunov results from \cite{KaJ11a}, and was extended to impulsive systems in \cite{BAB24}. The proof of the converse Lyapunov result for the unboundedness observability property in \cite{AnS99} uses this strategy as well, first formulating an auxiliary system with a noisy \emph{output injection}, and proving a converse Lyapunov result for \revblue{RFC} of this auxiliary system in the second step.


In a similar fashion, integral input-to-state stability (iISS) is characterized by iISS Lyapunov functions in \cite{ASW00}. 
For input-output-to-state stability (IOSS), uniform output-to-state stability of an auxiliary system is employed to derive \revblue{corresponding IOSS Lyapunov functions} \cite{KSW01}. For input-to-output stability (IOS), robust output stability is used to generate IOS converse Lyapunov theorems \cite{SoW00,InW01} and, in the more recent approach via an auxiliary system with a \emph{regularized output} \cite{BDM25}, the strong robust output stability property is a necessary and sufficient condition for the existence of a vector IOS Lyapunov function.

The need to formulate the auxiliary system with a noisy state feedback makes it necessary to analyze the well-posedness of this system, and \revblue{impose additional requirements on the original system to make this analysis work.} If the system is given in an abstract form, then formulating such an auxiliary system \revblue{can also be challenging}, as there are no equations of motion. 
\emph{In this work, we \revblue{take a different approach and give} a direct proof of the converse Lyapunov theorem for the BRS property. This is possible thanks to the new stability and regularity notions for systems with trajectory-dominated inputs.}

The detailed description of our strategy is given in Section~\ref{sec:Philosophy-converse-Lyapunov-theorems}. 
We hope that this method can be used to develop the converse Lyapunov theorems for other stability notions, most notably for ISS and its extensions, which will be an interesting topic for future research.

\textbf{Notation:}
In the sequel, the nonnegative integers are denoted by $\N_0$, the natural numbers by $\N$, the real numbers by $\R$, the nonnegative real numbers by $\R_+$ and the open balls of radius $C$ around zero in \revblue{normed vector} spaces $X$, $U$ and $\U$, respectively, by $B_C$, $B_{C,U}$ and $B_{C,\U}$. For a subset $\Omega$ of a \revblue{normed vector} space, we denote its closure by $\overline{\Omega}$. 

We call a map $f\colon X \to Y$ between two \revblue{normed vector} spaces $X,Y$ \emph{Lipschitz continuous on bounded sets}, if for each $C > 0$, there exists $L_C > 0$ such that for all $x_1,x_2 \in B_C$, it holds that
\begin{align*}
    \norm{f(x_1) - f(x_2)}_Y 
    &\leq L_C \norm{x_1 - x_2}_X.
\end{align*}
\revblue{For $V\subset X$, $W\subset Y$, we denote the space of continuous maps $f\colon V \to W$ by $\Cont(V,W)$.}
Moreover, let
\begin{align*}
    \operatorname{Lip}_1
    &\coloneq \revblue{\{\gamma\colon \R_+\to\R_+ \mid}\\
    &\qquad\revblue{|\gamma(r)-\gamma(s)|\leq|r-s|
    \text{ for all }r,s\in\R_+\}.}
\end{align*}

On finite-dimensional spaces, the notion of Lipschitz continuity on bounded sets and \emph{local Lipschitz continuity} are equivalent \cite[Proposition 1.51]{Mir23}. On infinite-dimensional spaces, however, Lipschitz continuity on bounded sets is a stronger property (see \cite[Rem. 4]{BAB24}).

We consider the standard classes of comparison functions (cf. \cite[p. xvi]{Mir23})
\begin{align*}
    \PP &\coloneq \{\gamma: \R_+ \to \R_+ \,|\, \gamma(0) = 0,\ \gamma \text{ is continuous},  \\
    &\hspace{4cm}\gamma(s) > 0 \text{ for all } s > 0\}, \\
    \KK &\coloneq \{\gamma \in \PP \,|\, \gamma \text{ is strictly increasing}\}, \\
    \KK_\infty &\coloneq \{\gamma\in \KK \,|\, \gamma \text{ is unbounded}\}, \\
    \LL &\coloneq \{\gamma\colon\R_+\to\R_+ \mid
    \gamma \text{ is continuous,}\\
    &\hspace{2.3cm}\revblue{\text{strictly decreasing, and }}
    \lim_{t\to\infty}\gamma(t)=0\}, \\
    \KK\LL &\coloneq
    \{\revblue{\beta\colon \R_+\times\R_+\to\R_+} \mid
    \beta(\ph,t)\in\KK,\ t\geq0,\\
    &\hspace{4cm}\beta(r,\ph)\in\LL,\ r>0\}.
\end{align*}

\revblue{For every Lebesgue-measurable $I\subset\R_+$ and every normed vector space $Z$, we denote by $\LL^\infty(I,Z)$ the Lebesgue space of strongly measurable, essentially bounded functions $f\colon I\to Z$, endowed with the norm $\norm{f}_{\LL^\infty(I,Z)}\coloneq\esssup_{t\in I}\norm{f(t)}_Z$.}

\section{Preliminaries}\label{sec:preliminaries}
\subsection{Abstract control systems}

\revblue{First, we define the notion of an abstract control system.}
\begin{definition}\label{def:controlSystem}
    Consider a triple $\Sigma = (X, \U, \phi)$ consisting of
    \begin{enumerate}
        \item a normed vector space $(X, \norm{\ph}_X)$, called the \emph{state space};
        \item a normed \emph{vector space of input values} $(U, \norm{\ph}_U)$ and the normed \emph{vector space of admissible inputs} $\U = \LL^\infty(\I, U)$, equipped with the essential supremum norm, which we denote by $\norm \ph_\U$.
        For $u_1,u_2 \in \U$, we denote the concatenation at time $t \in \I$
        by
        \begin{align*}
            u_1 \diamond_t u_2 (s)
            = \begin{cases}
                u_1(s), &\text{if } s \in [0,t),\\
                u_2(s-t), &\text{if } s \geq t;
            \end{cases}
        \end{align*}
        \item\label{def:controlSystemProperty4} a map $\phi\colon D_\phi \to X$, $D_\phi \subset \I \times X \times \U$, called \emph{\revblue{flow} map}, so that for all $(x,u) \in  X \times \U$ it holds that 
        \begin{center}
            $D_\phi \cap (\I \times \{(x,u)\}) = [0,t_{\max}) \times \{(x,u)\}$, 
        \end{center}
        for a certain $t_{\max} = t_{\max}(x,u) \in (0, + \infty]$. The corresponding interval $[0,t_{\max})$ is called the \emph{maximal domain of definition} of the mapping $t \mapsto \phi(t, x, u)$, which we call a \emph{trajectory} of the system.
    \end{enumerate}
    
    The triple $\Sigma$ is called a \emph{(control) system} if it satisfies the following axioms.
    \begin{enumerate}
        \renewcommand\theenumi{($\Sigma$\arabic{enumi})}
        \item\label{cond:identityProperty} \emph{Identity property}: For all $(x,u) \in X \times \U$, it holds that $\phi(0, x, u) = x$.
        \item\label{cond:causalityProperty} \emph{Causality}: For all $(t,x,u) \in D_\phi$ and all $\widetilde u \in \U$ such that \revblue{$u= \widetilde u$ on $[0,t]$}, it holds that $[0,t] \times \{(x, \widetilde u)\} \subset D_\phi$ and $\phi(t,x,u) = \phi(t,x,\widetilde u)$.
        \item\label{cond:continuityProperty} \emph{Continuity}: For each $(x,u) \in X \times \U$ the map $t \mapsto \phi(t,x,u)$ is continuous on its maximal domain of definition.
        \item\label{cond:cocycleProperty} \emph{Cocycle property}: For all $x \in X$, $u \in \U$ and $t,s \geq 0$ 
        \revblue{such that $t+s \in[0,t_{\max}(x,u))$}, we have 
        \begin{align*}
            \phi(t + s, x, u) &= \phi\!\paren{s, \phi(t, x, u), u(\ph + t)}.
        \end{align*}
    \end{enumerate}
\end{definition}

\begin{remark}
    This class of control systems is similar to that considered in, e.g., \cite{MiW18b,Mir23e,BDM26} and comprises ODEs, evolution partial differential equations, sufficiently regular boundary control systems \cite[Sec. 4.1]{MiP20}, \cite{Sch20}, switched systems, time-delay systems, and broad classes of well-posed linear systems.
    
    As opposed to \cite{BDM26}, for the development of the Lyapunov theory, we restrict the system class to continuous-time systems and require $\U = \LL^\infty(\I, U)$, and continuity of trajectories.
    
    \revblue{The results in this article can be transferred to $\U$ being the space of bounded piecewise continuous functions over $U$ endowed with the supremum norm} with a very similar proof.
\end{remark}

\begin{definition}\label{def:forwardComplete}
    We call a control system $\Sigma = (X,\U,\phi)$ \emph{forward complete (FC)}, if for each $x \in X$, $u \in \U$ and $t \in \I$, the value $\phi(t,x,u) \in X$ is well-defined.
\end{definition}

\revblue{
When forward completeness of control systems is not known or given, the following property allowing to prolong bounded trajectories, plays an important role:
\begin{definition}\label{def:BIC}
    We say that a control system $\Sigma = (X,\U,\phi)$ satisfies the \emph{boundedness-implies-continuation (BIC) property}, if for each $x \in X$ and $u \in \U$, the boundedness of the maximal solution $\phi(\cdot,x,u)$ implies $t_{\max}(x,u)=+\infty$.
\end{definition}
}

\subsection{Regularity of solutions}

We proceed to the main concept of this article.
\begin{definition}[\!\!\cite{MiW18b}]\label{def:BRS}
    A forward complete control system $\Sigma$ is said to have \emph{bounded \revblue{(finite-time)} reachability sets (BRS)} if for all \revblue{$C>0$ and $\tau>0$} it holds that
	\begin{align*} 
		\revblue{\sup_{\norm{x}_X \le C,\,\norm{u}_\U\le C,\,t\in[0,\tau]}
        \norm{\phi(t,x,u)}_X<\infty.}
	\end{align*}
\end{definition}

\revblue{
We also introduce the related concept for a given \revblue{nonempty} set of inputs $\D \subset \U$:
\begin{definition}\label{def:RFC}
    \revblue{A control system $\Sigma$ is said to be \emph{robustly forward complete (RFC) with respect to inputs in $\D$} if $\Sigma$ is forward complete and} for all $C>0$ and $\tau>0$,
    it holds that
	\begin{align*} 
		\revblue{\sup_{\norm{x}_X \le C,\,d\in\D,\,t\in[0,\tau]}
        \norm{\phi(t,x,d)}_X<\infty.}
	\end{align*}
\end{definition}
}

Inspired by the approach in \cite{SoW95} for the construction of ISS Lyapunov functions, we introduce the following notions:
\begin{definition}\label{def:TDI}
    Let $\Sigma$ be a forward complete control system and $\eta \in \KK_\infty$. For a given $x \in X$, we call the set
    \begin{align*}
        \U_x^\eta \coloneq \Big\{v \in\U\,|
        \norm{v(\theta)}_U 
        \leq \eta\!\paren{\norm{\phi(\theta,x,v)}_X} \text{ for a.e. } \theta \ge 0\Big\}\!,    
    \end{align*}
    the \emph{set of trajectory-dominated inputs with respect to initial value $x \in X$ and growth margin $\eta$}. We will use the colloquial term \emph{set of trajectory-dominated inputs} whenever $x$ and $\eta$ are clear from the context.
\end{definition}
Note that for any $x\in X$ and $\eta\in\Kinf$ we have that $0 \in \U_x^\eta$, and thus $\U_x^\eta \neq\emptyset$.

\revblue{Next, we introduce a novel notion of robust forward completeness. Thereby, we modify the notion of \revblue{RFC} introduced in \cite{Mir23e} \revblue{to adapt it to inputs of the form} $\U_x^\eta$. As we will see, this property is implied by BRS, and is a good starting point for the proof of converse Lyapunov results for BRS.}
\begin{definition}\label{def:RFCwrtTDI}
    A forward complete control system $\Sigma$ is said to be \emph{robustly forward complete \revblue{(RFC)} with respect to trajectory-dominated inputs} if there exists $\eta \in \KK_\infty$ such that for all \revblue{$C>0$ and $\tau>0$},
    it holds that
	\begin{align*} 
		\revblue{\sup_{\norm{x}_X\le C,\,u\in\U_x^\eta,\,t\in[0,\tau]}
        \norm{\phi(t,x,u)}_X<\infty.}
	\end{align*}
\end{definition}
\revblue{Roughly speaking, a system is \revblue{RFC} with respect to trajectory-dominated inputs if the trajectories can be bounded in norm uniformly in the norm of the initial state $x \in X$ for finite times and all inputs $\U_x^\eta$. Compared to classical RFC, the set of admissible inputs in RFC with respect to trajectory-dominated inputs depends on the trajectories evolving from these inputs.}


\section{Reachability and
behavior for trajectory-dominated inputs}\label{sec:BRSandTDI}

\revblue{We start by proving quantitative characterizations} of BRS and \revblue{RFC} with respect to trajectory-dominated inputs, which will help in the proof of our main result and are useful in their own right.

We call a function $h\colon \R_+^n \to \R_+$, $n \in \N$\revblue{,} \emph{\revblue{componentwise nondecreasing}} if $(r_1, \dots, r_n) \leq (R_1, \dots, R_n)$ implies that $h(r_1, \dots, r_n) \leq h(R_1, \dots, R_n)$, where we use the component-wise partial order defined by the cone $\R_+^n$. 
\begin{proposition}[Quantitative characterizations of BRS]
\label{prop:BRScharacterization}
    Let $\Sigma$ be a forward complete control system. The following are equivalent:
    \begin{enumerate}
        \item\label{cond:BRScharacterization1} $\Sigma$ has BRS.
        \item\label{cond:BRScharacterization2} There exists a continuous and \revblue{componentwise nondecreasing} function $\mu\colon \R_+^3 \to \R_+$ such that for all $x \in X$, $u \in \U$ and $t \in \I$, it holds that
        \begin{align*}
            \norm{\phi(t,x,u)}_X \leq\mu\!\paren{\norm{x}_X\!, \norm{u}_\U\!,t}\!.
        \end{align*}
        \item\label{cond:BRScharacterization3} There exist $\chi_1,\chi_2,\chi_3 \in \revblue{\KK_\infty}$ and a constant $c\geq 0$ \revblue{such that, for all $x \in X$, $u \in \U$, and $t \in \I$, we have}
        \begin{align*}
            \norm{\phi(t,x,u)}_X &\leq \chi_1(t) + \chi_2\!\paren{\norm{x}_X} + \chi_3\!\paren{\norm{u}_\U} + c. 
        \end{align*}
    \end{enumerate}
\end{proposition}
\begin{proof}
    The equivalence \ref{cond:BRScharacterization1}$\iff$\ref{cond:BRScharacterization2} is taken from \cite[Lem.~3]{MiW18b}. 
    
    \textit{\ref{cond:BRScharacterization2}$\implies$\ref{cond:BRScharacterization3}:} \revblue{Define $h(s) \coloneq \mu(s,s,s)-\mu(0,0,0)$ and $\chi_1(s)=\chi_2(s)=\chi_3(s) \coloneq h(s)+s$ for $s\geq0$. Since $\mu$ is continuous and \revblue{componentwise nondecreasing}, $h$ is continuous and nondecreasing with $h(0)=0$. Consequently, each $\chi_i$ belongs to $\KK_\infty$.} With $c \coloneq \mu(0,0,0)$, we obtain
    \begin{align*}
        &\norm{\phi(t,x,u)}_X 
        \leq \mu\!\paren{\norm{x}_X\!, \norm{u}_\U\!,t} - \mu(0,0,0) + \mu(0,0,0) \\
        &\leq \max\argbraces{\mu\!\paren{\norm{x}_X\!,\! \norm{x}_X\!,\!\norm{x}_X\!}\!, \!\mu\!\paren{\norm{u}_\U\!, \!\norm{u}_\U\!,\!\norm{u}_\U\!}\!,\! \mu\!\paren{t, t, t}} \\
        &\qquad- \mu(0,0,0) + \mu(0,0,0) \\
        &\leq \chi_1(t) + \chi_2\!\paren{\norm{x}_X} + \chi_3\!\paren{\norm{u}_\U} + c.
    \end{align*}
    
    \textit{\ref{cond:BRScharacterization3}$\implies$\ref{cond:BRScharacterization2}:} The statement follows from the definition $\mu(q,r,s) \coloneq \chi_1(s) + \chi_2\!\paren{q} + \chi_3\!\paren{r} + c$.
\end{proof}

\revblue{The definition of BRS is purely \emph{qualitative}, requesting that finite-time reachability sets are bounded, but not saying how large they may grow. In contrast,  Proposition~\ref{prop:BRScharacterization} provides a \emph{quantitative} formulation of the BRS property
 giving information on how fast the reachability sets grow and clarifying the influence of initial conditions, inputs, and time on the size of a reachability set.}

Similarly, we characterize \revblue{RFC} with respect to trajectory-dominated inputs.
\begin{proposition}
\label{prop:RFCcharacterization}
\emph{(Quantitative criterion of \revblue{RFC} with respect to trajectory-dominated inputs)} 
    Let $\Sigma$ be a forward complete control system \revblue{and $\eta \in \KK_\infty$}.
    \revblue{Then the following are equivalent:}
    \begin{enumerate}
        \item\label{cond:RFCcharacterization1} $\Sigma$ is \revblue{RFC} with respect to trajectory-dominated inputs with growth margin $\eta$.
        \item\label{cond:RFCcharacterization2} There exists a continuous and \revblue{component-wise nondecreasing} function $\nu\colon \R_+^2 \to \R_+$ such that for all $x \in X$, $u \in \U_x^\eta$, and $t \in \I$, it holds that
        \begin{align*}
            \norm{\phi(t,x,u)}_X \leq\nu\!\paren{\norm{x}_X\!,t}\!.
        \end{align*}

        \revblue{\item\label{cond:RFCcharacterization4} There exist $\kappa \in \KK_\infty\cap\operatorname{Lip}_1$ and a constant $c \geq 0$ so that for all $x \in X$, $u \in \U_x^\eta$, and $t \in \I$, it holds that
        \begin{align}
        \label{eq:RFC-kappa-characterization}
            \norm{\phi(t,x,u)}_X \leq\kappa^{-1}\!\paren{t + \norm{x}_X + c}\!.%
        \end{align}%
        }%
    \end{enumerate}%
\end{proposition}

\begin{proof}
    \textit{\ref{cond:RFCcharacterization1}$\iff$\ref{cond:RFCcharacterization2}:} The proof is analogous to \cite[Lem. 2.12]{MiW19a}. In particular, $\eta$ can be chosen the same in both statements.

    \textit{\ref{cond:RFCcharacterization2}$\implies$\ref{cond:RFCcharacterization4}:} \revblue{Set $h(s) \coloneq \nu(s,s)-\nu(0,0)$ and $\overline{\nu}(s) \coloneq h(s)+s$. Since $\nu$ is continuous and \revblue{component-wise nondecreasing}, $h$ is continuous and nondecreasing with $h(0)=0$, and therefore $\overline{\nu}\in\KK_\infty$. Choose $\kappa \in \KK_\infty \cap \operatorname{Lip}_1$ such that
    \[
    \kappa(s) \leq \overline{\nu}^{-1}(s).
    \]
    Such $\kappa$ exists by \cite[Lem. A.18]{Mir23}.} Setting $c \coloneq \nu(0,0) \geq 0$, and noting that \revblue{$h(s)\leq\overline{\nu}(s)\leq \kappa^{-1}(s)$}, $s\ge0$, we obtain
    \begin{align}
        \hspace{-4mm}\nu(\norm{x}_X\!&,t) 
        \leq \nu(\norm{x}_X + t,\norm{x}_X + t) \nonumber\\
        &= \revblue{h(\norm{x}_X + t)} + \nu(0,0) \leq \kappa^{-1}(\norm{x}_X + t) + c
        \label{ineq:RFCcomponentwiseBoundedByUniformFunction}
    \end{align}
    \revblue{As $\kappa \in \operatorname{Lip}_1$,} 
    we have that
    \begin{align*}
        s + \kappa(r) - \kappa(s) \leq r, \quad \text{for all } r \geq s \geq 0.
    \end{align*}
    Now by the substitution $r \coloneq \kappa^{-1}(\norm{x}_X + t + c)$ and $s \coloneq \kappa^{-1}(\norm{x}_X + t)$, we have the inequality
    \begin{align*}
        &\kappa^{-1}(\norm{x}_X + t) + c \\
        &\qquad= \kappa^{-1}(\norm{x}_X + t) + \kappa(\kappa^{-1}(\norm{x}_X + t + c)) \\
        &\qquad\qquad- \kappa(\kappa^{-1}(\norm{x}_X + t)) \\
        &\qquad\leq \kappa^{-1}(\norm{x}_X + t + c).
    \end{align*}
    Combining this estimate with  \eqref{ineq:RFCcomponentwiseBoundedByUniformFunction}, we see that for all $x \in X$, $u \in \U_x^{\revblue{\eta}}$ and all $t \in \I$, it holds that
        \begin{align*}
            \norm{\phi(t,x,u)}_X \leq \nu(\norm{x}_X\!,t) \leq\kappa^{-1}\!\paren{t + \norm{x}_X + c}\!.
        \end{align*}    
    \textit{\ref{cond:RFCcharacterization4}$\implies$\ref{cond:RFCcharacterization2}:} \revblue{This follows directly from the definition} $\nu(r,s) \coloneq \kappa^{-1}\!\paren{s + r + c}$ for all $r,s \geq 0$.
\end{proof}


\revblue{In Proposition~\ref{prop:BRScharacterization}, we have established a quantitative criterion for BRS. Though useful, it does not help us directly with the proof of converse BRS Lyapunov theorems, as the estimate which is given for the flow by this proposition depends on the norm of the input, and a Lyapunov function is a function of a state only. In contrast, Proposition~\ref{prop:RFCcharacterization} provides an input-independent upper bound for the flow of systems that are \revblue{RFC} with respect to trajectory-dominated inputs. The next lemma shows that BRS systems have the latter property. It will be ultimately helpful in the upcoming section.}


\begin{lemma}\label{lem:BRStoRFC}
    \revblue{If $\Sigma$ is a BRS control system, then it is \revblue{RFC} with respect to trajectory-dominated inputs.}

    Moreover, the growth margin $\eta \in \KK_\infty$ for \revblue{RFC} with respect to trajectory-dominated inputs can be chosen such that $\eta \in \operatorname{Lip}_1$.
\end{lemma}

\begin{proof}
    Let $\Sigma$ be BRS. Then, by Proposition \ref{prop:BRScharacterization}\ref{cond:BRScharacterization3}, 
    there exist $\chi_1, \chi_2, \chi_3 \in \revblue{\KK_\infty}$ and $c \geq 0$ such that for each $x \in X$, $u \in \U$, and $t \in \I$, it holds that
    \begin{align*}
        \norm{\phi(t,x,u)}_X &\leq \chi_1(t) + \chi_2\!\paren{\norm{x}_X} + \chi_3\!\paren{\norm{u}_\U} + c. 
    \end{align*}
    Let $\alpha \in \KK_\infty$ be such that 
    \begin{align*}
        \alpha(s) &= 4\max\argbraces{s, \chi_1(s), \chi_2(s), \chi_3(s)}, \quad s \in \R_+,
    \end{align*}    
    and $\eta \in \KK_\infty\cap \operatorname{Lip}_1$ such that $\eta \leq \frac{1}{2}\alpha^{-1}$. Such $\eta$ exists by \cite[Lem. A. 18]{Mir23}.

    Then, it holds that
    \begin{align}
        &2\eta\!\paren{\norm{\phi(t,x,u)}_X} 
        \leq 2\eta\!\paren{\chi_1(t) + \chi_2\!\paren{\norm{x}_X} + \chi_3\!\paren{\norm{u}_\U} + c} \nonumber \\
        &\qquad\leq \alpha^{-1}\!\paren{\chi_1(t) + \chi_2\!\paren{\norm{x}_X} + \chi_3\!\paren{\norm{u}_\U} + c} 
        \nonumber \\
        &\qquad\leq \alpha^{-1}\!\paren{4\max\argbraces{\chi_1(t), \chi_2\!\paren{\norm{x}_X}\!, \chi_3\!\paren{\norm{u}_\U}\!, c}}
        \nonumber \\
        &\qquad\leq t + \norm{x}_X + \norm{u}_\U + c. \label{ineq:estimateTildeAlpha1ofphi}
    \end{align}
    Let $x \in X$, $u \in \U_x^\eta$ and $t \revblue{> 0}$.
    Then, from \eqref{ineq:estimateTildeAlpha1ofphi} and due to causality,
    we obtain
    \begin{align*}
        &2\eta(\|\phi(t,x,u)\|_X) 
        = 
        2\eta\!\paren{\norm{\phi(t,x,u \diamond_t 0)}_X} \\
        &\qquad\leq t + \norm{x}_X + \norm{u \diamond_t 0}_\U + c \\ 
        &\qquad= t + \norm{x}_X + \esssup_{s \in [0,t]}\argbraces{\norm{u(s)}_U} + c \\ 
        &\qquad\leq t + \norm{x}_X + \esssup_{s \in [0,t]}\argbraces{\eta\!\paren{\norm{\phi(s,x,u)}_X}} + c 
    \end{align*}
    \revblue{As $\phi(\ph,x,u)$ is continuous by \ref{cond:continuityProperty}, taking the  supremum over $[0,t]$ in the last inequality  results in
    \begin{align*}
        2 \cdot \sup_{s \in [0,t]}&\eta\!\paren{\norm{\phi(s,x,u)}_X} \\
        &\qquad\leq t + \norm{x}_X + \sup_{s \in [0,t]}\eta\!\paren{\norm{\phi(s,x,u)}_X} + c,
    \end{align*}
    and we obtain
    \begin{align*}
        \eta\!\paren{\norm{\phi(t,x,u)}_X} 
        &\leq \sup_{s \in [0,t]}\eta\!\paren{\norm{\phi(s,x,u)}_X} \leq  t + \norm{x}_X + c.
    \end{align*}
    Overall, we have
    }
    \begin{align*}
        \norm{\phi(t,x,u)}_X
        &\leq  \eta^{-1}\!\paren{t + \norm{x}_X + c} < \infty
    \end{align*}
    for all $x \in X$, $u \in \U_x^\eta$ and $t \in \R_+$ and by Proposition~\ref{prop:RFCcharacterization}, $\Sigma$ is \revblue{RFC} with respect to trajectory-dominated inputs.
\end{proof}

\section{Lyapunov criterion}\label{sec:LyapunovBRS}

\subsection{Formulation of the Lyapunov criterion}

Let us recall the notion of \emph{Lipschitz continuity on compact intervals} as defined in \cite[Def. 3.22]{Mir24} as follows:

\begin{definition}\label{def:continuityOnCompactIntervals}
\revblue{Let $\Sigma$ be a forward complete control system.} 
    The \emph{flow of $\Sigma$ is called Lipschitz continuous on compact intervals} if \revblue{there is a componentwise nondecreasing map $L: \R_+\times\R_+\to(0,+\infty)$ so that for each $\tau > 0$, each $C > 0$,} for any $x_1, x_2 \in B_C$, $u \in B_{C,\U}$ and $t < \tau$, we have
    \begin{align*}
        \norm{\phi(t,x_1,u) - \phi(t,x_2,u)}_X \leq L(\tau,C)\norm{x_1 - x_2}_X.
    \end{align*}
\end{definition}

We suggest the following related concept, which requires Lipschitz continuity to hold over trajectories under \emph{different inputs}. 

\begin{definition}\label{def:continuityOnCompactIntervalswrtInputs}
    Let $\Sigma$ be a forward complete control system and $\eta \in \KK_\infty$. 
    We say that the \emph{flow of $\Sigma$ is Lipschitz continuous on compact intervals with respect to trajectory-dominated inputs \revblue{(with margin $\eta$)}} if 
    \revblue{there exists a componentwise nondecreasing map $L: \R_+\times\R_+\to(0,+\infty)$ such that for each $\tau > 0$, each $C > 0$,all} $x_1, x_2 \in B_C$, and $u_1 \in \U_{x_1}^\eta$, there exists $u_2 \in \U_{x_2}^\eta$, such that for all $t \le \tau$, it holds that
    \begin{align}\label{ineq:LipschitzContinuityOnCompactIntervalswrtInputs}
        \norm{\phi(t,x_1,u_1) - \phi(t,x_2,u_2)}_X \leq L(\tau,C)\norm{x_1 - x_2}_X\!.
    \end{align}
\end{definition}
\revblue{We note that the previous property is symmetric in $x_1, x_2$ since the roles of the two can be swapped.}

\revblue{Although at first glance the idea of choosing $u_2$ to satisfy the Lipschitz inequality for a given $u_1$ seems exotic, it is natural from the viewpoint of set-valued analysis.

Let $\Sigma$ be a forward complete control system and
$\eta\in \Kinf$. For $\tau>0$ and $x\in X$, define a \emph{solution multifunction}
\begin{align*}
    \mathcal S_\tau^\eta(x)
    \coloneq
    \big\{\phi(\ph,x,u)|_{[0,\tau]} \mid
    u\in\U_x^\eta\big\}
    \subset\Cont([0,\tau],X).
\end{align*}
For nonempty subsets $A,B\subset\Cont([0,\tau],X)$,
define the \emph{extended Hausdorff (pseudo)distance}  between $A$ and $B$ by
\begin{align*}
    d_H(A,B)\coloneq\max\!\big\{\!
    &\sup_{z\in A}\inf_{w\in B}\norm{z-w}_\infty\!{,} \sup_{w\in B}\inf_{z\in A}\norm{z-w}_\infty
    \!\big\}\!,
\end{align*}
where $\|\cdot\|_\infty$ is the supremum norm on $\Cont([0,\tau],X)$.

Lipschitz regularity introduced in Definition~\ref{def:continuityOnCompactIntervalswrtInputs} can be equivalently characterized as
\begin{proposition}\label{prop:HausdorffCharacterizationOfFlowRegularity}
The following statements are equivalent:
\begin{enumerate}
    \item\label{cond:HausdorffFlowRegularity1}
    The flow of $\Sigma$ is Lipschitz continuous on compact
    intervals with respect to trajectory-dominated inputs
    with margin $\eta$.
    \item\label{cond:HausdorffFlowRegularity2}
    There exists a componentwise nondecreasing map
    $L:\R_+\times\R_+\to(0,+\infty)$ such that for every
    $\tau>0$, $C>0$, and $x_1,x_2\in B_C$,
    \begin{align*}
        d_H\big(\mathcal S_\tau^\eta(x_1),
                \mathcal S_\tau^\eta(x_2)\big)
        \leq L(\tau,C)\norm{x_1-x_2}_X.
    \end{align*}
\end{enumerate}
\end{proposition}

\begin{proof}
The sets $\mathcal S_\tau^\eta(x)$ are nonempty because
$0\in\U_x^\eta$.

\noindent\textit{\ref{cond:HausdorffFlowRegularity1}$\implies$\ref{cond:HausdorffFlowRegularity2}:}
Fix $\tau,C>0$, $x_1,x_2\in B_C$, and
$z_1\in\mathcal S_\tau^\eta(x_1)$.
Choose $u_1\in\U_{x_1}^\eta$ so that $z_1=\phi(\cdot,x_1,u_1)|_{[0,\tau]}$.
Definition~\ref{def:continuityOnCompactIntervalswrtInputs}
provides $u_2\in\U_{x_2}^\eta$ such that for all $t \le \tau$ we have
\[
    \norm{\phi(t,x_1,u_1)-\phi(t,x_2,u_2)}_X
    \leq L(\tau,C)\norm{x_1-x_2}_X.
\]
Consequently, with
$z_2=\phi(\ph,x_2,u_2)|_{[0,\tau]}$, we obtain
\[
    \norm{z_1-z_2}_\infty
    \leq L(\tau,C)\norm{x_1-x_2}_X.
\]
Taking the infimum over $z_2$ and then the supremum over
$z_1$ bounds the first directed distance. Interchanging
$x_1$ and $x_2$ bounds the other one, proving the claim.

\noindent
\textit{\ref{cond:HausdorffFlowRegularity2}$\implies$\ref{cond:HausdorffFlowRegularity1}:}
Fix $\tau,C>0$, $x_1,x_2\in B_C$, and
$u_1\in\U_{x_1}^\eta$.
If $x_1=x_2$, take $u_2=u_1$.
Otherwise, put
$z_1=\phi(\ph,x_1,u_1)|_{[0,\tau]}$.
The Hausdorff estimate implies
\[
    \inf_{z_2\in\mathcal S_\tau^\eta(x_2)}
    \norm{z_1-z_2}_\infty
    \leq L(\tau,C)\norm{x_1-x_2}_X.
\]
As $L(\tau,C)\norm{x_1-x_2}_X>0$, there is $z_2\in\mathcal S_\tau^\eta(x_2)$
such that
\[
    \norm{z_1-z_2}_\infty
    <2L(\tau,C)\norm{x_1-x_2}_X.
\]
Thus, \eqref{ineq:LipschitzContinuityOnCompactIntervalswrtInputs} holds with  $u_2\in\U_{x_2}^\eta$: $z_2 = \phi(\cdot,x_2,u_2)|_{[0,\tau]}$, and $2L$ instead of $L$.
\end{proof}
}

\begin{remark}\label{rem:relationOfLipschitzContinuityNotions}
    Obviously, Lipschitz continuity on compact intervals does not imply Lipschitz continuity on compact intervals with respect to trajectory-dominated inputs in the general case. 
    The converse relation does not hold either, as we will show in Proposition \ref{prop:LipschitzonCIwrtTDIdoesNotImplyLipschitzonCI}.
\end{remark}

\revblue{
\begin{remark}\label{rem:Margin-reduction-for-RFC-andLipschitz-cont}
Let $\xi,\eta \in \Kinf$. If $\xi\leq \eta$ pointwise, then RFC of $\Sigma$ w.r.t. $\eta$-dominated inputs implies RFC of $\Sigma$ w.r.t. $\xi$-dominated inputs. 
At the same time, if the flow of $\Sigma$ is Lipschitz continuous on compact intervals with respect to trajectory-dominated inputs with margin $\eta$, 
this does not imply that $\Sigma$ is Lipschitz continuous on compact intervals with respect to trajectory-dominated inputs with margin $\xi$. 
Indeed, while for any  $x_1, x_2 \in B_C$, and any $u_1 \in \U_{x_1}^\xi$, there exists $u_2 \in \U_{x_2}^\eta$ such that 
\eqref{ineq:LipschitzContinuityOnCompactIntervalswrtInputs} holds, but 
it is not necessarily that $u_2 \in \U_{x_2}^\xi$.
\end{remark}
}

For a continuous function $V: X \to \R$, we define the \emph{Lie derivative} $\dot V_u(x)$ as the \emph{upper right-hand Dini derivative} at zero for the function $t \mapsto V\!\paren{\phi(t,x,u)}$ \cite[Def. 2.9]{Mir23}, i.e.,
\begin{align*}
    \dot V_u(x)
    &\coloneq \limsup_{s \searrow 0} \tfrac{1}{s}\paren{V\!\paren{\phi(s,x,u)} - V(x)},
\end{align*}
\revblue{where the values of $\dot V_u(x)$ are in the extended real line.}

Let us recall the key concept for this talk introduced in \cite[Proposition 4.4]{Mir23e}.
\begin{definition}\label{def:BRSLyapunovFunction}
    Let $\Sigma$ be a control system. We call a function $V\in \Cont\!\paren{X,\R_+}$ a \emph{BRS Lyapunov function} for system $\Sigma$ if it satisfies the following conditions:
	\begin{enumerate}
		\item\label{def:BRSLyapunovFunctionCondition1} There exist $\alpha_1, \alpha_2 \in \KK_\infty$ and $C \geq 0$ such that
		\begin{align*}
			\alpha_1\!\paren{\norm{x}_X} &\leq V(x) \leq \alpha_2\!\paren{\norm{x}_X} + C, \quad x \in X.
		\end{align*}
		\item\label{def:BRSLyapunovFunctionCondition2*} There exists $\chi \in \KK_\infty$ such that for all $x \in X$ and $u \in \U$, it holds that
        \begin{align*}			
            \chi\!\paren{\norm u_\U} &\leq
            \norm x_X 
            \qquad\implies
            \qquad \dot V_u(x) \leq V\!\paren{x}.
        \end{align*}
	\end{enumerate}
\end{definition}

In \cite[Proposition 4.4]{Mir23e}, it was shown that \revblue{for systems satisfying the BIC property}, the existence of a BRS Lyapunov function implies BRS. \emph{Our main contribution in this work is to show the converse implication for systems that are Lipschitz continuous on compact intervals with respect to trajectory-dominated inputs:}
\revblue{
\begin{theorem}
\label{thm:BRSLyapunovFunction}
    

        
    Let $\Sigma$ be a control system satisfying the BIC property. 
    Consider the following statements:
    \begin{enumerate}

        
        \item\label{cond:BRSLyapunovFunction5} There exists a Lipschitz continuous on bounded sets BRS Lyapunov function for $\Sigma$.
        \item\label{cond:BRSLyapunovFunction4} There is a (continuous) BRS Lyapunov function for $\Sigma$.
        \item\label{cond:BRSLyapunovFunction1} $\Sigma$ is BRS.
        \item\label{cond:BRSLyapunovFunction6} There exists $\eta\in\Kinf$ such that $\Sigma$ is RFC with respect to trajectory-dominated inputs with growth margin $\eta$.
    \end{enumerate}
    Then the following implications hold:
    \begin{center}
        \ref{cond:BRSLyapunovFunction5}$ \implies $\ref{cond:BRSLyapunovFunction4}$ \implies $\ref{cond:BRSLyapunovFunction1}$ \implies $\ref{cond:BRSLyapunovFunction6}.  
    \end{center}
    Furthermore, if the flow of $\Sigma$ is Lipschitz continuous on compact intervals with respect to trajectory-dominated inputs with the same margin $\eta$ as in \ref{cond:BRSLyapunovFunction6}, then 
    \begin{center}
        \ref{cond:BRSLyapunovFunction6}$ \implies $\ref{cond:BRSLyapunovFunction5}.  
    \end{center}
\end{theorem}
}

\begin{proof}
    \revblue{The outline of the proof is depicted in Figure \ref{fig:BRSimplications}.
    The implication \ref{cond:BRSLyapunovFunction4}$\implies$\ref{cond:BRSLyapunovFunction1} is shown in \cite[Proposition 4.4]{Mir23e} and \ref{cond:BRSLyapunovFunction5}$\implies$\ref{cond:BRSLyapunovFunction4} is clear. The remaining implication will be proven in Lemma~\ref{lem:RFCtoBRSLF}.} 
\end{proof}

\begin{figure}[htbp]
    \centering
    \begin{tikzpicture}[node distance=2cm]
        \node(char5) at (0,0) {\ref{cond:BRSLyapunovFunction5}};
        \node[right of = char5] (char4){\ref{cond:BRSLyapunovFunction4}};
        \node[right=2.6cm of char4] (char1){\ref{cond:BRSLyapunovFunction1}};
        \node[right=2.0cm of char1] (char6){\ref{cond:BRSLyapunovFunction6}};
        \node(dummy2) at (0,-1) {};

        \path
        (char5) edge[thick,double,double equal sign distance,-{Implies[]}] node[anchor=south]{\footnotesize Clear} (char4)
        (char4) edge[thick,double,double equal sign distance,-{Implies[]}] node[anchor=south]{\footnotesize \cite[Proposition 4.4]{Mir23e}} (char1)
        (char1) edge[thick,double,double equal sign distance,-{Implies[]}] node[anchor=south]{\footnotesize Lemma \ref{lem:BRStoRFC}} (char6);
        \draw[transparent]  (char6) -- coordinate(dummy1) (char5);
        \draw [thick,double,double equal sign distance,-{Implies[]},rounded corners=2mm] 
        (char6) |-(dummy1|-dummy2) -| node[pos=0,anchor=south]{\footnotesize Lemma \ref{lem:RFCtoBRSLF}} node[pos=0,anchor=north,align=center]{\footnotesize $+$ the flow of $\Sigma$ is Lipschitz continuous on compact intervals with\\ \footnotesize respect to trajectory-dominated inputs with the same margin $\eta$ as in \ref{cond:BRSLyapunovFunction6}} (char5);
    \end{tikzpicture}
    \caption{Outline of the proof of Theorem \ref{thm:BRSLyapunovFunction}.}
    \label{fig:BRSimplications}
\end{figure}
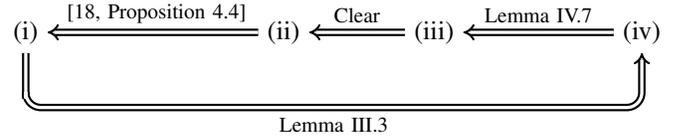

\subsection{The philosophy behind the proof}
\label{sec:Philosophy-converse-Lyapunov-theorems}

\revblue{
Let us discuss the philosophy behind our proof of the converse Lyapunov result. It is divided into the following steps:
\begin{enumerate}
    \item\label{item:BRSintuition1} \emph{Qualitative formulation of BRS.} 
    
    \item\label{item:BRSintuition2} \emph{Quantitative reformulation of BRS} (Proposition~\ref{prop:BRScharacterization}). 

    \item\label{item:BRSintuition3} \emph{Quantitative reformulation of BRS tailored to the proof of converse Lyapunov results}. This is the concept of RFC with respect to trajectory-dominated inputs, together with  Proposition~\ref{prop:RFCcharacterization}, and Lemma~\ref{lem:BRStoRFC}.

    \item\label{item:BRSintuition4} \emph{Choosing the right regularity of the flow.}

    \item\label{item:BRSintuition5} \emph{Define Lyapunov function-like building blocks.}

    \item\label{item:BRSintuition6} \emph{Construct BRS Lyapunov function.}
\end{enumerate}

Getting \ref{item:BRSintuition2} is important, as this reformulation gives us the \q{gains}, characterizing the bounds of reachability sets. However, these bounds are dependent on the norm of the input fed to the system, and it is hard to exploit them directly to construct a BRS Lyapunov function. 
Hence, we need a Step \ref{item:BRSintuition3}, where we prove RFC of a system with respect to trajectory-dominated inputs, giving us the \emph{bound \eqref{eq:RFC-kappa-characterization} on the flow  characterized by $\kappa \in \Kinf$, which is independent of the inputs with a growth margin $\eta$}.

In Step \ref{item:BRSintuition4}, we need to assume a certain regularity of the flow. The better the regularity properties of the flow are, the better regularity of the Lyapunov function we can achieve. This is elaborated in detail in 
\cite[Section 3.4]{KaJ11a}. 

Here, we require Lipschitz continuity on compact intervals with respect to trajectory-dominated inputs. Moreover, we require that the margin $\eta$ is the same as for RFC with respect to trajectory-dominated inputs, so we have the same class of inputs for both properties.

In Step \ref{item:BRSintuition5}, we define in \eqref{eq:definitionU_q} the building blocks $U_q$ for our future BRS Lyapunov function. Here, we exploit the function $\kappa$ constructed in Step \ref{item:BRSintuition3}, which allows us to obtain the upper bounds for $U_q$ independent of the inputs.
We finish the proof by constructing $V$ from $(U_q)_{q\in\N}$ in \eqref{eq:BRS-candidate}.
In fact, with this strategy, we obtain BRS Lyapunov functions which are Lipschitz continuous on bounded sets.

This informal strategy is \q{universal}, and its versions have been used for the proof of other converse Lyapunov results. 
As an example, consider the proof of the converse Lyapunov result for uniform global asymptotic stability (UGAS) of ODE systems as in \cite[Theorem B.31]{Mir23}. 

One may start (Step \ref{item:BRSintuition1}) with a qualitative formulation of UGAS, such as stability $+$ (uniform) global attractivity. Then (Step \ref{item:BRSintuition2}) one gets a quantitative restatement of UGAS as a $\KL$-estimate. Afterward (Step \ref{item:BRSintuition3}), one uses Sontag's lemma \cite{Son98} to get a more structured and more useful $\KL$-bound, etc.

We expect that this philosophy can be used to give a direct proof of the converse ISS Lyapunov theorem for infinite-dimensional systems.
}

\subsection{Detailed proof}

In the following, we use the family of functions defined by
\begin{align*}
    G_q(z) \coloneq\max\argbraces{0,z-\tfrac{1}{q}},\quad q \in \N, \quad z \in \R_+.
\end{align*}
Note that for every $q \in \N$, we have $G_q \in \operatorname{Lip}_1$, i.e., 
\begin{align}
    \abs{G_q(z_1) - G_q(z_2)} \leq \abs{z_1 - z_2},\quad z_1,z_2\in\R_+, \label{ineq:estimateG_k1}
\end{align}
and for every $q \in \N$\revblue{,} $a \geq 1$\revblue{, and $z \in \R_+$}, it holds that
\begin{align}
    G_q(az) \leq aG_q(z) + \tfrac{a - 1}{q}. \label{ineq:estimateG_k2}
\end{align}
For the proof, see \cite[Lemma 3.2]{Mir23e}.

For our next result, we need the following technical lemma \revblue{from \cite[Lem. B.29]{Mir23}}:
\begin{lemma}\label{lem:boundSupremumDifference}
    Let $J$ be \revblue{a non-empty} index set and $v,w\colon J \to \R$. Then, the inequality
    \begin{align*}
        \sup\nolimits_{s \in J}\braces{v(s) - w(s)}
        &\geq\sup\nolimits_{s \in J}\!\braces{v(s)} - \sup\nolimits_{s \in J}\!\braces{w(s)}
    \end{align*}
    holds \revblue{if the right-hand side is well-defined, i.e., if not both $\sup\nolimits_{s \in J}\!\braces{v(s)}$ and $\sup\nolimits_{s \in J}\!\braces{w(s)}$ are infinite}.
\end{lemma}

We show that \revblue{RFC} with respect to trajectory-dominated inputs implies the existence of a BRS Lyapunov function.

\begin{lemma}
\label{lem:RFCtoBRSLF}
    Let $\Sigma$ be a \revblue{forward complete} control system,
    and the flow of $\Sigma$ be Lipschitz continuous on compact intervals with respect to trajectory-dominated inputs with margin $\eta \in \KK_\infty$. If $\Sigma$ is \revblue{RFC} with respect to trajectory-dominated inputs \revblue{with margin $\eta$}, then there exists a BRS Lyapunov function for $\Sigma$,  \revblue{which is Lipschitz continuous on bounded sets}.
\end{lemma}

\begin{proof}
    Let $\Sigma$ be Lipschitz continuous on compact intervals with respect to trajectory-dominated inputs $\U_x^\eta$ and \revblue{RFC} with respect to trajectory-dominated inputs \revblue{with margin $\eta$}, where we assume that $\kappa \in \KK_\infty \revblue{\cap \operatorname{Lip}_1}$ and \revblue{$c \geq 0$} are as in Proposition~\ref{prop:RFCcharacterization}\ref{cond:RFCcharacterization4}. 
    
    \emph{Step 1.1: Defining pre-Lyapunov functions.}
    Let us design the building blocks for the BRS Lyapunov function that will be constructed later. For any $x \in X$ and $q \in \N$, we define 
    \begin{align}
    \label{eq:definitionU_q}
        &U_q(x) 
        \coloneq \sup_{v \in\U_x^\eta} \sup_{s\geq 0} G_q \big(e^{-s} \revblue{\kappa}\!\paren{\norm{\phi(s,x,v)}_X}\big) \\
        &= \sup \!\Big\{\sup_{s\geq 0} \!\Big\{\max\!\Big\{0, e^{-s} \revblue{\kappa}\!\paren{\norm{\phi(s,x,v)}_X} - \tfrac 1 q \Big\}\Big\} \nonumber \\
        &\quad\,\Big|\, v \in \U \colon \text{for a.e. } \theta \in \R_+\colon \norm{v(\theta)}_U 
        \leq \eta\!\paren{\norm{\phi(\theta,x,v)}_X}\Big\}.\nonumber 
    \end{align}
    
    We define for $R \geq 0$ and $q \in \N$
    \begin{align*}
        \Theta = \Theta(R,q) \coloneq \min\argbraces{t \in [1,\infty)\,|\, e^{-t}\paren{t + R + c} \leq \tfrac 1 q}.
    \end{align*}
    
    Note that the map $t \mapsto te^{-t}$ is strictly decreasing for $t \geq 1$. Therefore, $\Theta$ is continuous and increasing in the first component and for all $R\geq 0$, $q\in\N$ and $t \geq \Theta(R,q)$, it holds that $e^{-t}\paren{t + R + c} \leq \tfrac 1 q$.

    Then, for any given \revblue{$R \ge 0$}, and all $x \in \overline{B_R}$, $u \in \U_x^\eta$, and $t \geq \Theta(R,q)$, 
    \revblue{\eqref{eq:RFC-kappa-characterization} gives}
    \begin{align}
        e^{-t}\revblue{\kappa}\!\paren{\norm{\phi(t,x,u)}_X}
        &\leq e^{-t}\paren{t + \norm{x}_X + c}
        \leq \tfrac 1 q. \label{ineq:UUOUqBoundedInTime}
    \end{align}

    By \eqref{ineq:UUOUqBoundedInTime}, for any $R \ge 0$ we can represent $U_q$ for $x\in X$: $\|x\|\leq R$ as
    \begin{align}
        U_q(x) 
        &= \sup_{v \in \U_x^\eta} \sup_{s \in [0,\Theta(R,q)]} \!G_q \big(e^{-s} \revblue{\kappa}\!\paren{\norm{\phi(s,x,v)}_X}\big).\label{eq:definitionU_qUUO}
    \end{align}
    Thus, it is sufficient to evaluate $e^{-s} \revblue{\kappa}\!\paren{\norm{\phi(s,x,v)}_X}$ on the compact time interval $[0,\Theta(R,q)]$.
    
    \textit{Step 1.2: Sandwich bounds for $U_q$:} By substitution of $v = 0$, $s = 0$ into the definition of $U_q(x)$, we obtain the lower bound and by \eqref{eq:RFC-kappa-characterization} the upper bound
    \begin{align}
        G_q(\revblue{\kappa}\!\paren{\norm{x}_X})
        \leq U_q(x) 
        \leq \Theta(\norm x_X\!,q) + \norm{x}_X + c. \label{ineq:sandwichboundsU_qUUO}
    \end{align}
    
    \textit{Step 1.3: Growth bound for $U_q$:} Next, we define $\chi \in \KK_\infty$, 
    \revblue{
    \begin{align*}
        \chi(s) \coloneq \eta^{-1}(2s), \quad s \in \R_+.     
    \end{align*}
    }
    If the inequality
    \begin{align}
    \label{eq:Lyap-Premise-BRS}
        \chi\!\paren{\norm u_\U} &\leq
        \norm x_X
    \end{align}
    holds for some $x \in X$ and $u \in \U$, then it follows that
    \begin{align*}
        \norm{u}_\U
        \leq \tfrac{1}{2} \eta\!\paren{\norm{x}_X}.
    \end{align*}
    \revblue{As $\phi(\ph,x,u)$ is continuous due to \ref{cond:continuityProperty}, there} exists \revblue{$\tau =\tau(x,u) >0$} such that for all $t \in [0,\tau)$, we have
    \begin{align}
        \norm{u}_\U
        \leq  \eta\!\paren{\norm{\phi(t,x,u)}_X}. \label{ineq:preliminarybounduyandx}
    \end{align}

    Hence, for $x \in X$ and $u \in \U$ satisfying \eqref{eq:Lyap-Premise-BRS}\revblue{,} for all $t \in [0,\tau(x,u))$, by definition of $U_q(x)$, we have
    \begin{align*}
        &U_q(\phi(t,x,u)) \\
        &= \sup_{v \in\U_{\phi(t,x,u)}^\eta} \sup_{s\geq 0} G_q \big(e^{-s} \revblue{\kappa}\!\paren{\norm{\phi(s,\phi(t,x,u),v)}_X}\big) \\
        &= \sup\! \Big\{\sup_{s\geq 0} \argbraces{G_q\argparen{e^{-s} \revblue{\kappa}\!\paren{\norm{\phi(s,\phi(t,x,u),v)}_X}}} \,\Big|\, v \in \U\colon \\
        &\quad \text{for a.e. } \theta \in \R_+\colon \norm{v(\theta)}_U
        \leq \eta\!\paren{\norm{\phi(\theta,\phi(t,x,u),v)}_X}\Big\}.
    \end{align*}
    By the cocycle property \ref{cond:cocycleProperty}, and since $(u \diamond_t v)(t + \theta) = v(\theta)$ for all $\theta\ge 0$, we proceed to
    \begin{align*}
        &U_q(\phi(t,x,u)) \\
        &\qquad= \sup\!\Big\{\sup_{s\geq 0} \argbraces{G_q\argparen{e^te^{-t -s} \revblue{\kappa}\!\paren{\norm{\phi(t + s,x, u \diamond_t v)}_X}}} \\
        &\qquad\quad\,\Big|\, v \in \U\colon \text{for a.e. } \theta \in \R_+\colon \\
        &\qquad\quad\quad\norm{(u \diamond_t v)(t + \theta)}_U
        \leq \eta\!\paren{\norm{\phi(t + \theta,x,u \diamond_t v)}_X}\Big\}
    \end{align*}
    
    We rewrite this by substituting $\theta^* \coloneq t + \theta$:
    \begin{align*}        
        &U_q(\phi(t,x,u)) \\
        &\qquad= \sup\!\Big\{\sup_{s\geq 0} \argbraces{G_q\argparen{e^te^{-t -s} \revblue{\kappa}\!\paren{\norm{\phi(t + s,x, u \diamond_t v)}_X}}} \\
        &\qquad\quad\,\Big|\, v \in \U\colon \text{for a.e. } \theta^* \in [t,\infty)\colon  \\
        &\qquad\qquad\quad\norm{(u \diamond_t v)(\theta^*)}_U
        \leq \eta\!\paren{\norm{\phi(\theta^*,x,u \diamond_t v)}_X}\Big\}.
    \end{align*}
    Finally, in view of \eqref{ineq:preliminarybounduyandx}, and since $t \in [0,\tau(x,u))$, we see that
    \begin{align*}        
        &U_q(\phi(t,x,u)) \\
        &\qquad= \sup\!\Big\{\sup_{s\geq 0} \argbraces{G_q\argparen{e^te^{-(t+s)} \revblue{\kappa}\!\paren{\norm{\phi(t + s,x, u \diamond_t v)}_X}}} \\
        &\qquad\quad\,\Big|\, v \in \U\colon \text{for a.e. } \theta^* \in \R_+\colon  \\
        &\qquad\qquad\quad\norm{(u \diamond_t v)(\theta^*)}_U
        \leq \eta\!\paren{\norm{\phi(\theta^*,x,u \diamond_t v)}_X}\Big\}.
    \end{align*} 
    
    We substitute $s^* \coloneq t + s$, and $v^* \coloneq u\diamond_t v$. The supremum cannot become smaller by allowing a larger set of times $s^*\geq 0$ instead of $s^* \geq t$ and a larger set of inputs $v^* \in \U$ instead of $v^* = u\diamond_t v$, $v \in\Uc$, i.e.,
    \begin{align}
        &U_q(\phi(t,x,u)) \nonumber\\
        &\leq \sup\!\Big\{\sup_{s^*\geq 0} \argbraces{G_q\argparen{e^te^{-s^*} \revblue{\kappa}\!\paren{\norm{\phi(s^*,x, v^*)}_X}}} \nonumber\\
        &\quad\,\Big|\, v^* \in \U\colon \text{for a.e. } \theta^* \in \R_+\colon  \nonumber\\
        &\qquad\qquad\qquad\qquad\quad\norm{v^*(\theta^*)}_U
        \leq \eta\!\paren{\norm{\phi(\theta^*,x,v^*)}_X}\Big\} \nonumber\\
        &= \sup_{v^* \in \U_x^\eta} \sup_{s^*\geq 0} \argbraces{G_q\argparen{e^te^{-s^*} \revblue{\kappa}\!\paren{\norm{\phi(s^*,x, v^*)}_X}}} \nonumber\\
        &\leq e^t \sup_{v^* \in \U_x^\eta} \sup_{s^*\geq 0} \argbraces{G_q\argparen{e^{-s^*} \revblue{\kappa}\!\paren{\norm{\phi(s^*,x, v^*)}_X}}} + \tfrac{e^t - 1}{q} \nonumber\\
        &= e^t U_q(x) + \tfrac{e^t - 1}{q}, \qquad \forall q \in \N, \label{ineq:boundU_q}
    \end{align}
    where we used \eqref{ineq:estimateG_k2} in the second last step.
    
    Then, for all $x \in X$ and $u \in \U$ satisfying \eqref{eq:Lyap-Premise-BRS}\revblue{,} 
    we have \revblue{for the Lie derivative}
    \begin{align}
        &\tfrac{\diff}{\diff t} (U_q)_u(x)
        = \revblue{\limsup}_{t \searrow 0} \tfrac{U_q(\phi(t,x,u)) - U_q(x)}{t} \nonumber\\
        &\leq \revblue{\limsup_{t \searrow 0}} \tfrac{1}{t}\paren{e^t U_q(x) + \tfrac{e^t - 1}{q} - U_q(x)} = U_q(x) + \tfrac{1}{q}. \label{ineq:boundDerivativeU_q}
    \end{align}
    
    \textit{Step 1.4: Lipschitz continuity on bounded balls of $U_q$:}
    Pick any $q \in \N$, $R>0$ and let $x_1, x_2 \in B_R$. Using \eqref{eq:definitionU_qUUO}, we have
    \begin{align}
        &\abs{U_q(x_1) - U_q(x_2)} \nonumber\\
        &= \max\argbraces{U_q(x_1) - U_q(x_2), U_q(x_2) - U_q(x_1)} \nonumber\\
        &= \max\!\bigg\{\sup_{u_1 \in \U_{x_1}^\eta} \sup_{s \in [0,\Theta(R,q)]} G_q \big(e^{-s} \revblue{\kappa}\!\paren{\norm{\phi(s,x_1,u_1)}_X}\big) \nonumber\\
        &\quad - \sup_{u_2 \in \U_{x_2}^\eta} \sup_{s \in [0,\Theta(R,q)]} G_q \big(e^{-s} \revblue{\kappa}\!\paren{\norm{\phi(s,x_2,u_2)}_X}\big),\nonumber\\
        &\quad\sup_{u_2 \in \U_{x_2}^\eta} \sup_{s \in [0,\Theta(R,q)]} G_q \big(e^{-s} \revblue{\kappa}\!\paren{\norm{\phi(s,x_2,u_2)}_X}\big) \nonumber \\
        &\quad- \sup_{u_1 \in \U_{x_1}^\eta} \sup_{s \in [0,\Theta(R,q)]} G_q \big(e^{-s} \revblue{\kappa}\!\paren{\norm{\phi(s,x_1,u_1)}_X}\big)\bigg\}. \label{eq:estimateLipschitzContinuityBRS}
    \end{align}
    Since the flow of $\Sigma$ is Lipschitz continuous on compact intervals with respect to trajectory-dominated inputs, there is a map $k_{x_1,x_2}\colon \U_{x_1}^\eta \to \U_{x_2}^\eta$ which maps each $u_1 \in \U_{x_1}^{\revblue{\eta}}$ to $u_2 \revblue{= k_{x_1,x_2}(u_1)} \in \U_{x_2}^{\revblue{\eta}}$ such that
    \begin{align*}
        \revblue{\norm{\phi(t,x_1,u_1) - \phi(t,x_2,u_2)}_X} \leq L(\Theta(R,q),R)\norm{x_1 - x_2}_X.
    \end{align*}
    holds for all $t \leq \Theta(R,q)$ (if there are multiple such $u_2$, the image may be selected arbitrarily).
    
    From the definition of $k_{x_1,x_2}$, Lemma \ref{lem:boundSupremumDifference} and \eqref{ineq:estimateG_k1}, it follows that
    \begin{align*}
        &\sup_{u_1 \in \U_{x_1}^\eta} \sup_{s \in [0,\Theta(R,q)]} G_q \big(e^{-s} \revblue{\kappa}\!\paren{\norm{\phi(s,x_1,u_1)}_X}\big)  \\
        &\qquad\quad- \sup_{u_2 \in \U_{x_2}^\eta} \sup_{s \in [0,\Theta(R,q)]} G_q \big(e^{-s} \revblue{\kappa}\!\paren{\norm{\phi(s,x_2,u_2)}_X}\big) \nonumber\\
        &\qquad\leq \sup_{u_1 \in \U_{x_1}^\eta} \bigg\{\sup_{s \in [0,\Theta(R,q)]} G_q \big(e^{-s} \revblue{\kappa}\!\paren{\norm{\phi(s,x_1,u_1)}_X}\big)  \\
        &\qquad\quad- \sup_{s \in [0,\Theta(R,q)]} G_q \big(e^{-s} \revblue{\kappa}\!\paren{\norm{\phi(s,x_2,k_{x_1,x_2}(u_1))}_X}\big)\bigg\} \nonumber\\
        &\qquad\leq \sup_{u_1 \in \U_{x_1}^\eta} \sup_{s \in [0,\Theta(R,q)]} \big|G_q \big(e^{-s} \revblue{\kappa}\!\paren{\norm{\phi(s,x_1,u_1)}_X}\big) \\
        &\qquad\qquad\qquad\qquad - G_q \big(e^{-s} \revblue{\kappa}\!\paren{\norm{\phi(s,x_2,k_{x_1,x_2}(u_1))}_X}\big)\big| \nonumber \\
        &\qquad\leq \sup_{u_1 \in \U_{x_1}^\eta} \sup_{s \in [0,\Theta(R,q)]} \big|e^{-s} \big(\revblue{\kappa}\!\paren{\norm{\phi(s,x_1,u_1)}_X}  \\
        &\qquad\qquad \qquad\qquad\qquad -  \revblue{\kappa}\!\paren{\norm{\phi(s,x_2,k_{x_1,x_2}(u_1))}_X}\big)\big|. \nonumber
    \end{align*}
    Furthermore, employing that $\revblue{\kappa} \in \operatorname{Lip}_1$, the reverse triangle inequality, and the assumption that the flow of $\Sigma$ is Lipschitz continuous on compact intervals with respect to trajectory-dominated inputs, we achieve 
    \begin{align}
        &\sup_{u_1 \in \U_{x_1}^\eta} \sup_{s \in [0,\Theta(R,q)]} G_q \big(e^{-s} \revblue{\kappa}\!\paren{\norm{\phi(s,x_1,u_1)}_X}\big) \nonumber\\
        &\quad- \sup_{u_2 \in \U_{x_2}^\eta} \sup_{s \in [0,\Theta(R,q)]} G_q \big(e^{-s} \revblue{\kappa}\!\paren{\norm{\phi(s,x_2,u_2)}_X}\big) \nonumber\\
        &\leq \sup_{u_1 \in \U_{x_1}^\eta} \sup_{s \in [0,\Theta(R,q)]} e^{-s} \big|\norm{\phi(s,x_1,u_1)}_X \nonumber\\
        &\quad- \norm{\phi(s,x_2,k_{x_1,x_2}(u_1))}_X\big| \nonumber \\
        &\leq \sup_{u_1 \in \U_{x_1}^\eta} \sup_{s \in [0,\Theta(R,q)]} \!\!\!\!\!\!e^{-s} \!\norm{\phi(s,x_1,u_1) - \phi(s,x_2,k_{x_1,x_2}(u_1))}_X\nonumber \\
        &\leq \sup_{u_1 \in \U_{x_1}^\eta} \sup_{s \in [0,\Theta(R,q)]} e^{-s} L(\Theta(R,q),R) \norm{x_1 - x_2}_X\nonumber \\
        &\leq L(\Theta(R,q),R) \norm{x_1 - x_2}_X. \label{ineq:continuityUUOU_q(x)}
    \end{align}
    From \eqref{eq:estimateLipschitzContinuityBRS}, \eqref{ineq:continuityUUOU_q(x)} and by exchanging the roles of $x_1$ and $x_2$, it follows for all $x_1,x_2 \in B_R$ that
    \begin{align}
        \abs{U_q(x_1) - U_q(x_2)}
        &\leq L(\Theta(R,q),R) \norm{x_1 - x_2}_X, \label{ineq:propertyG2UUO} 
    \end{align}
    i.e., $U_q$ is Lipschitz continuous on bounded balls.

    \textit{Step 2: Construction of $V$:} 
    Define $M\colon  \R_+ \times \N \to \R_+$ by
    \[
    M(R,q) \coloneq \max\argbraces{L(\Theta(R,q),R), \Theta(R,q)}.
    \]
    
    By the definition of $L$ and $\Theta$, we have
    \begin{align}
        q \geq R &\implies M(q,q) \geq M(R,q) \label{ineq:propertyG1UUO},
    \end{align}
    for all $q \in \N$ and $R \in \R_+$.

    We define a BRS Lyapunov function candidate by
    \begin{align}
    \label{eq:BRS-candidate}
        V(x) \coloneq 1 + \sum_{q = 1}^\infty \frac{2^{-q} U_q(x)}{1 + M(q,q)}, \quad x \in X.
    \end{align}
    Let us check that $V$ is indeed a BRS Lyapunov function that is \revblue{Lipschitz continuous on bounded sets. On this path, we will also show that the limit \eqref{eq:BRS-candidate} is well-defined in the sense that the sum converges uniformly on bounded balls on $X$.}

    \textit{Sandwich bounds for $V$:}
    \revblue{We verify Definition \ref{def:BRSLyapunovFunction} Condition \ref{def:BRSLyapunovFunctionCondition1}, i.e., w}e employ \eqref{ineq:sandwichboundsU_qUUO} to define $\alpha_1,\alpha_2 \in \KK_\infty$ and $C \geq 0$ such that
    \begin{align}
        \alpha_1\!\paren{\norm{x}_X} 
        \leq V(x) 
        \leq \alpha_2\!\paren{\norm x_X} + C, \label{ineq:sandwichBoundVUUO}
    \end{align}
    is true. Indeed, we set
    \begin{align*}
        \alpha_1(s) &\coloneq \sum_{q = 1}^\infty 2^{-q} \frac{G_q(\revblue{\kappa}(s))}{1 + M(q,q)},\quad s \geq 0.
    \end{align*}
    Clearly, $\alpha_1(0)=0$, $\alpha_1$ is increasing and as \revblue{$ \kappa \in\Kinf$}, so is $\alpha_1$. 
    For the upper bound, we define $\alpha_2\in \KK_\infty$ such that 
    \begin{align*}
        \alpha_2(s) \geq  s + \sum_{q = 1}^{\lfloor s \rfloor} 2^{-q} \tfrac{\Theta\paren{s,q}}{1 + M(q,q)},
    \end{align*}
    with the convention that the sum over an empty index set equals zero and $\lfloor\ph\rfloor$ is the floor function.
    
    Furthermore, we set $C \coloneq 2 + c$.
    \revblue{Then, using the estimate
    \begin{align*}
        \sum_{q = \lfloor\norm{x}_X \rfloor + 1}^\infty 2^{-q} \tfrac{\Theta\paren{\norm{x}_X,q}}{1 + M(q,q)}
        &\leq \sum_{q = \lfloor\norm{x}_X \rfloor + 1}^\infty 2^{-q} \tfrac{\revblue{M\paren{\norm{x}_X,q}}}{1 + M(q,q)} \\
        &\leq \sum_{q = \lfloor\norm{x}_X \rfloor + 1}^\infty 2^{-q} \tfrac{M(q,q)}{1 + M(q,q)} 
        \leq 1,
    \end{align*}
    where we made use of the definition of $M$ \revblue{and \eqref{ineq:propertyG1UUO}},
    and the bound for a geometric series, we achieve the upper bound}
    \begin{align*}
        V(x) &= 1 + \sum_{q = 1}^\infty \tfrac{2^{-q} U_q(x)}{1 + M(q,q)} \\
        &\leq 1 + \sum_{q = 1}^\infty 2^{-q} \tfrac{\Theta(\norm x_X,q) + \norm{x}_X + c}{1 + M(q,q)} \\
        &= 1 + \sum_{q = 1}^{\lfloor \norm{x}_X \rfloor} 2^{-q} \tfrac{\Theta\paren{\norm{x}_X,q}}{1 + M(q,q)} 
        + \sum_{q = \lfloor \norm{x}_X \rfloor + 1}^\infty \!\!2^{-q} \tfrac{\Theta\paren{\norm{x}_X,q}}{1 + M(q,q)} \\
        &\quad+ \sum_{q = 1}^\infty 2^{-q} \tfrac{\norm{x}_X + c}{1 + M(q,q)}\\
        &\leq \alpha_2\argparen{\norm{x}_X} + C\revblue{.}
    \end{align*}
    
    \textit{Growth bound for $V$:} \revblue{Whenever $\chi(\|u\|_{\U}) \leq \|x\|_X$, the finite-time estimate \eqref{ineq:boundU_q} gives, for all sufficiently small $t>0$ and every $q\in\N$,
    \[
        U_q(\phi(t,x,u))\leq e^tU_q(x)+\tfrac{e^t-1}{q}.
    \]
    Multiplying by $\tfrac{2^{-q}}{1+M(q,q)}$ and summing the convergent nonnegative series yields
    \begin{align*}
        &V(\phi(t,x,u)) \\
        &\leq 1+e^t\sum_{q=1}^\infty \frac{2^{-q}U_q(x)}{1+M(q,q)}
        +(e^t-1)\sum_{q=1}^\infty\frac{2^{-q}}{q(1+M(q,q))}\\
        &\leq \revblue{1+e^t(V(x)-1) + (e^t-1) = e^tV(x),}
    \end{align*}
    since $\sum_{q=1}^\infty \revblue{\frac{2^{-q}}{q(1+M(q,q))}}\leq1$. Hence, after subtracting $V(x)$, dividing by $t$, and taking the upper right-hand limit, $\dot V_u(x)\leq V(x)$.} Thus, the implication in Definition \ref{def:BRSLyapunovFunction}\ref{def:BRSLyapunovFunctionCondition2*} holds. 
    
    \textit{Lipschitz continuity of $V$:} From the definition of $V$, \eqref{ineq:propertyG2UUO} and \eqref{ineq:propertyG1UUO}, we get for any $R>0$ and any $x_1,x_2 \in B_R$ the Lipschitz bound
    \begin{align*}
        &\abs{V(x_1) - V(x_2)}
        \leq \sum_{q = 1}^\infty 2^{-q} \frac{\abs{U_q(x_1) - U_q(x_2)}}{1 + M(q,q)} \\
        &\qquad\leq \norm{x_1 - x_2}_X \paren{\sum_{q = 1}^\infty 2^{-q} \frac{M(R,q)}{1 + M(q,q)}}\\
        &\qquad\leq \norm{x_1 - x_2}_X \Bigg(1 + \sum_{q = 1}^{\lceil R \rceil} 2^{-q} \frac{M(R,q)}{1 + M(q,q)}\Bigg)\!,
    \end{align*}
    where $\lceil \ph \rceil$ is the ceiling function, that is, $V$ is Lipschitz continuous on bounded sets.

    \revblue{In particular, we showed that the sum \eqref{eq:BRS-candidate} converges uniformly on bounded balls on $X$ as the limit function is Lipschitz continuous on bounded balls and $V$ converges in $x = 0$ due to monotonicity and boundedness \eqref{ineq:sandwichBoundVUUO}.}
\end{proof}

Lemma \ref{lem:RFCtoBRSLF} closes the loop and accomplishes the proof of Theorem \ref{thm:BRSLyapunovFunction}.

\revblue{
\begin{remark}
The main novelty in the proof of Lemma \ref{lem:RFCtoBRSLF} is the usage of the classes of trajectory-dominated inputs combined with the transformation of BRS to RFC with respect to trajectory-dominated inputs. 
We adopt several other ingredients from the paper \cite{Mir23e}, which served as our main motivation. 
The usage of the regularizing functions $G_q$ as well as a series construction of $V$ is classical \revblue{and ensures a certain regularity (Lipschitz continuity on bounded sets) of the Lyapunov function}; see, e.g., the proof of \cite[Theorem 4.2.1]{Hen81}, as well as \cite[Section 3.4]{KaJ11a}. 
\end{remark}
}

For convenience, we state a counterpart to \cite[Corollary 2.12]{AnS99}.

\begin{corollary}\label{cor:variationBRSLyapunovFunction}
    Let $\Sigma$ be a \revblue{control system satisfying the BIC property. Let} there exist a function $W\in \Cont\!\paren{X,\R_+}$ satisfying the following conditions:
	\begin{enumerate}
		\item\label{def:variationBRSLyapunovFunctionCondition1} There exist $\alpha_1, \alpha_2 \in \KK_\infty$ and $C \geq 0$ such that
		\begin{align*}
			\alpha_1\!\paren{\norm{x}_X} &\leq W(x) \leq \alpha_2\!\paren{\norm{x}_X} + C, \quad x \in X.
		\end{align*}
		\item\label{def:variationBRSLyapunovFunctionCondition2*} There exists $\chi \in \KK_\infty$ such that for all $x \in X$ and $u \in \U$, it holds that
        \begin{align*}			
            \chi\!\paren{\norm u_\U} &\leq
            \norm x_X 
            \quad\implies
            \quad \dot W_u(x) \leq 1.
        \end{align*}
	\end{enumerate}
    \revblue{Then, $\Sigma$ is BRS.

    Conversely, if $\Sigma$ is BRS, and there exists $\eta \in \KK_\infty$ such that $\Sigma$ is RFC with respect to trajectory-dominated inputs with growth margin $\eta$ and the flow of $\Sigma$ is Lipschitz continuous on compact intervals with respect to trajectory-dominated inputs with the same margin $\eta$, then there exists a function $W\in \Cont\!\paren{X,\R_+}$ satisfying \ref{def:variationBRSLyapunovFunctionCondition1} and \ref{def:variationBRSLyapunovFunctionCondition2*}.}
\end{corollary}

\begin{proof}
    \revblue{If $W$ satisfies the assumptions of the Corollary, then $V:= W+1$ satisfies the sandwich estimate with $C+1$ instead of $C$, and 
    for all $x \in X$ and $u \in \U$, it holds that
    \begin{align*}			
        \chi\!\paren{\norm u_\U} &\leq
        \norm x_X 
        \quad\implies
        \quad \dot V_u(x) = \dot W_u(x) \leq 1 \leq V(x).
    \end{align*}
    Then, applying Theorem \ref{thm:BRSLyapunovFunction} proves the direct result.

    The converse result} is achieved by Theorem \ref{thm:BRSLyapunovFunction}, the choice $W \coloneq \ln(1 + V)$ and the chain rule for the Dini derivative \cite[Lem. A.30]{Mir23}.
\end{proof}

\section{
\revblue{Semilinear} evolution equations}\label{sec:BRSEvolutionEquations}

In the previous section, we showed that for systems, which are Lipschitz continuous on compact intervals with respect to trajectory-dominated inputs, the BRS property is equivalent to the existence of a BRS Lyapunov function and \revblue{RFC} with respect to \revblue{trajectory-dominated inputs, provided the regularity
and RFC properties hold with the same margin.} However, it is not a priori clear how restrictive the assumption of Lipschitz continuity on compact intervals with respect to trajectory-dominated inputs is. Therefore, we show that this condition is fulfilled for a broad class of \revblue{semilinear evolution equations} $\PDESigma$ which are BRS.


To this aim, we define \emph{\revblue{semilinear evolution equations}} \revblue{on Banach spaces $X,U$}
by
\begin{align}
\label{eq:EE-system}
    \PDESigma\colon \dot x(t) &= Ax(t) + f\!\paren{x(t),u(t)}, \quad t \in \I,
\end{align}
\revblue{where the closed linear operator $A: D(A) \to X$ is the infinitesimal generator of a \revblue{strongly continuous ($\Cont_0$\babelhyphen{nobreak})}semigroup $(T(t))_{t \geq 0}$ on $X$, and $f \in \Cont(X \times U, X)$.} 

We suppose that for each $x_0 \in X$ and $u \in \U \coloneq \LL^\infty(\I, U)$, a unique \revblue{maximal} mild solution of $\PDESigma$ exists, i.e., \revblue{there exists some $t_{\max} = t_{\max}(x_0,u) \in (0, + \infty]$ such that} a function $x \in \Cont([0,\revblue{t_{\max})},X)$ satisf\revblue{ies} for 
all $t \in [0,\revblue{t_{\max})}$ the integral equation
\begin{align}\label{eq:mildSolution}
    x(t) &=
    T(t)x_0 + \int_0^t T(t - s) f\!\paren{x(s),u(s)} \diff s.
\end{align}
We denote this solution by $\phi = \phi(t,x_0,u)$.

A sufficient condition for uniqueness of mild solutions is Lipschitz continuity on bounded sets of $f$ \cite[Thm. 3.7]{Mir24}.

\subsection{Sufficient conditions for Lipschitz continuity on compact intervals with respect to trajectory-dominated inputs}
The following theorem provides sufficient conditions for Lipschitz continuity of the flow on compact intervals with respect to trajectory-dominated inputs.
\revblue{Particularly, BRS evolution equations with the bounded-set Lipschitz condition \eqref{ineq:LipschitzConditionOnf} have the common-margin regularity required in Theorem~\ref{thm:BRSLyapunovFunction}.}

Let $\eta \in\KK_\infty$\revblue{. Furthermore, we define throughout this section} 
\begin{center}
$\D \coloneq \overline{B_{1,\U}}$,\qquad $D \coloneq \overline{B_{1,U}}$.    
\end{center}
 Consider the auxiliary closed-loop system
\begin{align}\label{eq:closedLoopSystem}
   \hspace{-2mm} \Sigma_\eta\colon \dot x(t) &= Ax(t) + f\!\paren{x(t),d(t) \eta(\|x(t)\|_X)},\ \ t \geq 0.
\end{align}
where $A$ and $f$ are as for $\PDESigma$ \revblue{and $d \in \D$}. 

\revblue{For $\Sigma_\eta$, RFC and uniform flow regularity refer to the restricted family $d\in\D$. In particular, RFC means that all solutions corresponding to $d\in \D$ are global and satisfy the bound in Definition~\ref{def:RFC}.} 

\revblue{W}e give sufficient conditions such that the solutions of $\PDESigma$ and $\Sigma_\eta$ are well-defined and have Lipschitz continuous flow.
\begin{theorem}
\label{thm:sufficientConditionsForLCwrtTDI}
\revblue{Let $\PDESigma$ be a forward-complete control system.}
Consider the following statements:
\begin{enumerate}
    \item\label{cond:sufficientConditionsForLCwrtTDI1}
    \revblue{$\PDESigma$ is BRS and $f$ is Lipschitz on bounded sets, i.e.,} for each $C > 0$, there exists $L_C > 0$ such that for all $x_1,x_2 \in B_C$, $u_1,u_2 \in B_{C,U}$, it holds that
    \begin{align}\label{ineq:LipschitzConditionOnf}
        \|f(x_1,u_1) &- f(x_2,u_2)\|_X\nonumber\\
        &\leq L_C \paren{\norm{x_1 - x_2}_X + \norm{u_1 - u_2}_U}\!.
    \end{align}
    \item\label{cond:sufficientConditionsForLCwrtTDI2}
    \revblue{$\PDESigma$ is BRS and f}or every $\eta \in \KK_\infty \cap \operatorname{Lip}_1$, the map $(x,d) \mapsto f(x,d\eta(\|x\|_\revblue{X}))$, is Lipschitz continuous on bounded balls \revblue{with respect to the first argument}, uniformly in $d$, namely: For each $C > 0$, there is $L_C > 0$ so that for all $x_1,x_2 \in B_C$, $d \in D$, we have
    \begin{align}
        \|f\argparen{x_1,d\eta\argparen{\norm{x_1}_X}} - f&\argparen{x_2,d\eta\argparen{\norm{x_2}_X}}\|_X \nonumber\\
                &\leq L_C \paren{\norm{x_1 - x_2}_X}\!. \label{ineq:LipschitzConditionOnf_Feedback}
    \end{align}
    \item\label{cond:sufficientConditionsForLCwrtTDI3}
    There exists $\eta \in \KK_\infty \cap \operatorname{Lip}_1$, such that $\PDESigma$ is \revblue{RFC} with respect to trajectory-dominated inputs with growth margin $\eta$ and for each $C > 0$, there is $L_C > 0$ so that for all $x_1,x_2 \in B_C$, $d \in D$,
    \eqref{ineq:LipschitzConditionOnf_Feedback} holds.
    
    \item\label{cond:sufficientConditionsForLCwrtTDI4} There exists $\eta \in \KK_\infty \cap \operatorname{Lip}_1$ such that the auxiliary closed loop system $\Sigma_\eta$ \revblue{has unique global mild solutions for all $x\in X$ and $d\in\D$, and} 
    is \revblue{RFC with respect to inputs in $\mathcal D$ and its flow is} Lipschitz continuous on compact intervals uniformly in $d \in \D$.
    \item\label{cond:sufficientConditionsForLCwrtTDI5} \revblue{There exists $\eta \in \KK_\infty \cap \operatorname{Lip}_1$, such that $\PDESigma$ is RFC with respect to trajectory dominated inputs with margin $\eta$ and} has a flow that is Lipschitz continuous on compact intervals with respect to trajectory-dominated inputs \revblue{with the same margin $\eta$}.
\end{enumerate}
Then the following implications hold:
\begin{center}
\ref{cond:sufficientConditionsForLCwrtTDI1}$ \implies $\ref{cond:sufficientConditionsForLCwrtTDI2}$ \implies $\ref{cond:sufficientConditionsForLCwrtTDI3}$ \implies $\ref{cond:sufficientConditionsForLCwrtTDI4}$ \implies $\ref{cond:sufficientConditionsForLCwrtTDI5}.  
\end{center}
\end{theorem}

\begin{remark}\label{rem:LipschitzConditionOnf}
    The Lipschitz condition on the regularity of $f$ in Statement \ref{cond:sufficientConditionsForLCwrtTDI1} is stronger than the one required for well-posedness, see, e.g., \cite[Thm. 3.17]{Mir24}. This stronger condition is imposed to achieve the regularity of the BRS Lyapunov function $V$, and it is natural in the context of converse Lyapunov theorems \cite[Section 2.6]{Mir23}, see also \cite[Section 2.12]{Mir23} for the discussion of regularity requirements for the right-hand side for various types of results in the ISS theory.
\end{remark}

\begin{proof}
    \emph{\ref{cond:sufficientConditionsForLCwrtTDI1}$ \implies $\ref{cond:sufficientConditionsForLCwrtTDI2}}: Let $\eta \in \KK_\infty \cap \operatorname{Lip}_1$. For each $C > 0$, each $d \in D$ and $x_1,x_2 \in B_C$, it holds that $d \eta(\norm{x_1}_X), d \eta(\norm{x_2}_X) \in B_{C,U}$ due to global Lipschitz continuity with Lipschitz constant 1 of $\eta$, i.e. 
    \[
    \eta(s) = \abs{\eta(s) - \eta(0)} \leq \abs{s - 0} = s,\quad s \in \R_+.
    \]
    Then, by Statement \ref{cond:sufficientConditionsForLCwrtTDI1} with corresponding $L_C$, the fact that $\eta \in \operatorname{Lip}_1$ and the reverse triangle inequality, we have that
    \begin{align*}
        &\|{f\argparen{x_1,d\eta\argparen{\norm{x_1}_X}} - f\argparen{x_2,d\eta\argparen{\norm{x_2}_X}}}\|_X \\
        &\qquad\leq L_C\paren{\norm{x_1 - x_2}_X + \norm{d\eta\argparen{\norm{x_1}_X} - d\eta\argparen{\norm{x_2}_X}}_U} \\
        &\qquad= L_C\paren{\norm{x_1 - x_2}_X + \norm{d}_U\abs{\eta\argparen{\norm{x_1}_X} - \eta\argparen{\norm{x_2}_X}}} \\
        &\qquad\leq L_C\paren{\norm{x_1 - x_2}_X + \abs{\norm{x_1}_X - \norm{x_2}_X}} \\
        &\qquad\leq 2L_C\norm{x_1 - x_2}_X.
    \end{align*}
    Hence, Statement \ref{cond:sufficientConditionsForLCwrtTDI2} holds.
    
    \emph{\ref{cond:sufficientConditionsForLCwrtTDI2}$ \implies $\ref{cond:sufficientConditionsForLCwrtTDI3}}: Let $\PDESigma$ be BRS. By Lemma \ref{lem:BRStoRFC}, there exists a growth margin $\eta \revblue{\in \KK_\infty \cap \operatorname{Lip}_1}$ such that $\PDESigma$ is \revblue{RFC} with respect to trajectory-dominated inputs \revblue{with growth margin $\eta$.}
    \revblue{The rest follows from Statement \ref{cond:sufficientConditionsForLCwrtTDI2}.}
    
    \emph{\ref{cond:sufficientConditionsForLCwrtTDI3}$ \implies $\ref{cond:sufficientConditionsForLCwrtTDI4}:\quad Wellposedness and BIC property of $\Sigma_\eta$:}  
    Let $\eta \in \KK_\infty \revblue{\cap \operatorname{Lip}_1}$ be the map from Statement \ref{cond:sufficientConditionsForLCwrtTDI3}. 
    By Statement \ref{cond:sufficientConditionsForLCwrtTDI3}, the map $x \mapsto f\argparen{x,d\eta\argparen{\norm{x}_X}}$ is Lipschitz continuous on bounded sets uniformly in $d \in D$. As $f$ is continuous, $(x,d) \mapsto f\argparen{x,d\eta\argparen{\norm{x}_X}}$ is continuous as well.     
    \revblue{Applying \cite[Thm. 3.7]{Mir24}, we see that an initial value problem for \eqref{eq:closedLoopSystem}, for any initial value $x \in X$ and any disturbance input $d \in \D$, has a unique maximal mild solution that we denote by $\phi_\CL(\ph,x,d)$. 
    \cite[Prop. 3.19]{Mir24} shows that $\Sigma_\eta$ has the BIC property.

    \emph{$\Sigma_\eta$ is RFC w.r.t. $\D$:}
    Let $R > 0$ be arbitrary, $x \in B_R$, $d \in \D$. 
    Let $[0,t_{\max,\eta}(x,d))$ denote the maximal domain of definition of $\phi_\CL(\cdot,x,d)$. Pick arbitrary $\tau \in [0,t_{\max,\eta}(x,d))$ and
    define $u$ by 
    \[
    u(t)=\begin{cases}
    d(t)\eta(\norm{\phi_\CL(t,x,d)}_X),&0\leq t<\tau,\\
    0,&t\geq\tau.
    \end{cases}
    \]
    Then $u \in \U$. Due to causality}, the function 
    \revblue{$\phi$ defined by \eqref{eq:closedLoopSystem} is the (unique) solution to $\PDESigma$ with input $u$ on $[0,\tau]$ and, moreover, $u \in \U_{x}^\eta$}. Due to \revblue{RFC of $\PDESigma$} with respect to trajectory-dominated inputs with a growth margin $\eta$, there exists a function $\nu$ as defined in Proposition~\ref{prop:RFCcharacterization}. \revblue{Hence, for $t \leq \tau$, we have
    \begin{align}
    \label{eq:RFC-Sigma-eta}
        \norm{\phi_\CL(t,x,d)}_X \leq \nu\argparen{\norm{x}_X,t}.
    \end{align}
    Assume that $t_{\max,\eta}(x,d)<+\infty$. 
    Since  $\tau \in [0,t_{\max,\eta}(x,d))$ is arbitrary, $\phi_\CL(\ph,x,d)$ is bounded on $[0,t_{\max,\eta}(x,d))$ and in view of the BIC property of $\Sigma_\eta$, $t_{\max,\eta}(x,d) = \infty$. Thus, $\Sigma_\eta$ is forward complete, and by \eqref{eq:RFC-Sigma-eta}, it 
    is RFC with respect to inputs in $\mathcal D$.}


    \emph{Regularity of the flow of $\Sigma_\eta$:}
    \revblue{Let $R > 0$ be arbitrary, $x_i \in B_R$, $d \in \D$, and $\phi_i(t)\coloneq\phi_\CL(t,x_i,d)$, $i \in \{1,2\}$. A}t time $s\le t \in \R_+$, the trajectory $\phi_i$ may be bounded by 
    \begin{align*}
        \norm{\phi_i(s)}_X
        \leq \nu(R,s)
        \eqcolon C_1(s) \le C_1(t).
    \end{align*}
    Moreover, we note that
    \begin{align*}
        &\Big\| d(t)\eta\!\paren{\norm{\phi_1(t)}_X} - d(t)\eta\!\paren{\norm{\phi_2(t)}_X}\Big\|_U \\
        &\qquad= \norm{d(t)}_U\abs{\eta\!\paren{\norm{\phi_1(t)}_X} - \eta\!\paren{\norm{\phi_2(t)}_X}} \\
        &\qquad\leq \abs{\eta\!\paren{\norm{\phi_1(t)}_X} - \eta\!\paren{\norm{\phi_2(t)}_X}} \\
        &\qquad\leq \abs{\norm{\phi_1(t)}_X - \norm{\phi_2(t)}_X} \\
        &\qquad\leq \norm{\phi_1(t) - \phi_2(t)}_X,
    \end{align*}
    for almost every $t \ge 0$, where we used that $\eta$ has the Lipschitz constant 1 and the reverse triangle inequality. Since $T$ defines a $\Cont_0$-semigroup, there exist $M> 0$ and $\lambda \ge 0$ such that $\norm{T(t)} \leq Me^{\lambda t}$ \cite[Thm. 2.2]{Paz83}, where $\norm{\ph}$ is the operator norm on $X$. 
    Then, it holds that
    \begin{align*}
        &\norm{\phi_1(t) - \phi_2(t)}_X \\
        &= \Bigg\|T(t)x_1 + \int_0^t T(t - s) f\!\paren{\phi_1(s),d(s)\eta\!\paren{\norm{\phi_1(s)}_X}} \diff s \\
        &\quad- T(t)x_2 - \int_0^t T(t - s) f\!\paren{\phi_2(s),d(s)\eta\!\paren{\norm{\phi_2(s)}_X}} \diff s \Bigg\|\nonumber\\
        &\leq \norm{T(t)}\norm{x_1 - x_2}_X \nonumber\\
        &+ \int_0^t \norm{T(t - s)}\|f\!\paren{\phi_1(s),d(s)\eta\!\paren{\norm{\phi_1(s)}_X}} \\
        &\qquad- f\!\paren{\phi_2(s),d(s)\eta\!\paren{\norm{\phi_2(s)}_X}}\|_X \diff s \nonumber\\
        &\leq Me^{\lambda t}\norm{x_1 - x_2} + \int_0^t Me^{\lambda (t - s)} \\
        &\quad \times L_{C_1(t)}\!\paren{\norm{\phi_1(s) - \phi_2(s)}_X\! + \norm{\phi_1(s) - \phi_2(s)}_X} \diff s \nonumber\\
        &\leq Me^{\lambda t}\norm{x_1 - x_2} \\
        &\quad+ \int_0^t 2Me^{\lambda (t - s)} L_{C_1(t)}\norm{\phi_1(s) - \phi_2(s)}_X  \diff s        .
    \end{align*}
    Above, we use that $C_1$ is increasing and, without loss of generality, we assume that the Lipschitz constant $L$ is increasing w.r.t. the radius of the ball. By the substitution $z_i(t) = e^{-\lambda t}\phi_i(t)$ and Grönwall's inequality, it follows that
    \begin{align*}
        &\norm{\phi(t,x_1,u_1) - \phi(t,x_2,u_2)}_X
        = \norm{\phi_1(t) - \phi_2(t)} \\
        &\qquad= e^{\lambda t}\norm{z_1(t) - z_2(t)}
        \leq Me^{(2M L_{C_1(t)} + \lambda)t}\norm{x_1 - x_2}.
    \end{align*}
    Hence, the flow of $\Sigma_\eta$ as defined in \eqref{eq:closedLoopSystem} is Lipschitz continuous on compact intervals uniformly in \revblue{$d \in \cal D$}.
    
    \emph{\ref{cond:sufficientConditionsForLCwrtTDI4}$ \implies $\ref{cond:sufficientConditionsForLCwrtTDI5}}: Let $\eta\in \KK_\infty \cap \operatorname{Lip}_1$ \revblue{be as in Statement \ref{cond:sufficientConditionsForLCwrtTDI4}}. 
    Pick any $R > 0$ and $\tau > 0$ and let $x_1, x_2 \in B_R$. We want to match each $u_1 \in \U_{x_1}^\eta$ with such a $u_2 \in \U_{x_2}^\eta$ that \eqref{ineq:LipschitzContinuityOnCompactIntervalswrtInputs} is satisfied for all $t \in [0,\tau)$. 
    Pick any $u_1 \in \U_{x_1}^\eta$. 
    As $\PDESigma$ is well-posed, there is a unique solution of this system with the initial condition $x_1$ and the input $u_1$, namely $\phi(\ph,x_1,u_1)$. 
    \revblue{Set $d(t)=u_1(t)/\eta(\norm{\phi(t,x_1,u_1)}_X)$ when the denominator is positive, and $d(t)=0$ otherwise. Domination forces $u_1(t)=0$ almost everywhere on the zero-denominator set. Thus $d$ is strongly measurable and $\norm{d(t)}_U\leq1$ a. e., so $d\in\D$ and}
    \begin{align}\label{eq:implicitDefinitiond}
        u_1(t) = d(t)\eta\!\paren{\norm{\phi(t,x_1,u_1)}_X} \text{ for a.e. } t \ge 0.    
    \end{align}
    We have
    \begin{align*}
        &\phi(t,x_1,u_1) 
        = T(t)x_1 + \int_0^t\! T(t - s) f\!\paren{\revblue{\phi(s,x_1,u_1)},u_1(s)} \diff s\\
        &= T(t)x_1 \\
        &\quad+ \int_0^t T(t - s) f\!\paren{\revblue{\phi(s,x_1,u_1)},\revblue{d(s)}\eta\!\paren{\norm{\revblue{\phi(s,x_1,u_1)}}_X}} \diff s.
    \end{align*}
    This shows that $t\mapsto \phi(t,x_1,u_1)$ is also a mild solution to $\Sigma_\eta$ as defined in \eqref{eq:closedLoopSystem} with initial condition $x(0)=x_1$ and the above $d$, i.e., $\phi(t,x_1,u_1) = \phi_{\CL}(t,x_1,d)$, $t\geq 0$. \revblue{Thus, the RFC bound for $\Sigma_\eta$ proves that $\PDESigma$ is RFC with respect to trajectory dominated inputs with margin $\eta$.}

    \revblue{For the chosen horizon $\tau$, define
    \begin{align}\label{eq:definitionOfu_2}
    u_2(t)=\begin{cases}
    d(t)\eta(\norm{\phi_\CL(t,x_2,d)}_X),&0\leq t<\tau,\\
    0,&t\geq\tau.
    \end{cases}
    \end{align}
    Continuity of the closed-loop trajectory on $[0,\tau]$ makes $u_2$ essentially bounded, hence $u_2\in\U$. Uniqueness gives $\phi(t,x_2,u_2)=\phi_\CL(t,x_2,d)$ for $t\leq\tau$. Domination holds before $\tau$ by this equality and after $\tau$ because $u_2=0$. Therefore $u_2\in\U_{x_2}^\eta$.}
    
    
    Hence, by the assumption that $\Sigma_\eta$ as defined in \eqref{eq:closedLoopSystem} is Lipschitz continuous on compact intervals uniformly in $d \in \D$, the flow of $\PDESigma$ is Lipschitz continuous on compact intervals with respect to trajectory-dominated inputs, i.e., for any $u_1 \in \U_{x_1}^\eta$, there exists $u_2 \in \U_{x_2}^\eta$ defined as above where $d$ is implicitly given in \eqref{eq:implicitDefinitiond}, such that for all $t \le \tau$, it holds that
    \begin{align*}
        &\norm{\phi(t,x_1,u_1) - \phi(t,x_2,u_2)}_X \nonumber\\
        &= \revblue{\norm{\phi_\CL(t,x_1,d) - \phi_\CL(t,x_2,d)}_X} 
        \leq L(\tau,R)\norm{x_1 - x_2}_X\!,
    \end{align*}
    where $L(\tau,R)$ is the Lipschitz constant of the flow of $\Sigma_\eta$.
\end{proof}

\begin{remark}
    Note that in the proof of Theorem \ref{thm:sufficientConditionsForLCwrtTDI}, we implicitly used the following fact: By \eqref{eq:implicitDefinitiond}, for each $x \in X$ and $u \in \U_x^\eta$, there exists a $d \in \D$ such that
    \begin{align*}
        \phi(\ph,x,u) = \phi_\CL(\ph,x,d).
    \end{align*}
    
    Conversely, by \eqref{eq:definitionOfu_2}, for each $x \in X$, $d \in \D$ and $\tau \in \R_+$, there exists $u \in \U_x^\eta$ such that for all $t \leq \tau$
    \begin{align*}
        \phi(t,x,u) = \phi_\CL(t,x,d).
    \end{align*}
    \revblue{This direction holds only on finite intervals, because $t\mapsto d(t)\eta(\norm{\phi_\CL(t,x,d)}_X)$ might not be essentially bounded on $\R_+$. The zero extension in \eqref{eq:definitionOfu_2} ensures belonging to  $\U$.}

    In total, this means that \revblue{there is a correspondence between inputs $u$ to system $\PDESigma$ and $d$ to system $\Sigma_\eta$ generating the same trajectories on finite intervals; this correspondence need not be one-to-one}, given that the regularity of these systems is high enough.
\end{remark}


\begin{corollary}\label{cor:RFCofCLSystem}
    We consider \revblue{the forward complete control system} $\PDESigma$
    \revblue{satisfying \eqref{ineq:LipschitzConditionOnf}}. Then, the following are equivalent:
    \begin{enumerate}
        \item\label{cond:RFCofCLSystem1} There exists $\eta \in \KK_\infty \revblue{\cap \operatorname{Lip}_1}$ such that $\Sigma_\eta$ as defined in \eqref{eq:closedLoopSystem} is \revblue{RFC} with respect to inputs in $\D$. 
        \item\label{cond:RFCofCLSystem2} \revblue{There exists $\eta \in \KK_\infty \cap \operatorname{Lip}_1$ such that} $\PDESigma$ is \revblue{RFC} with respect to trajectory-dominated inputs \revblue{with margin $\eta$}. 
        \item\label{cond:RFCofCLSystem3} $\PDESigma$ is BRS. 
    \end{enumerate}
\end{corollary}


    
    


\begin{proof}
    \revblue{By the implication \ref{cond:sufficientConditionsForLCwrtTDI1}$\implies$\ref{cond:sufficientConditionsForLCwrtTDI2} of Theorem \ref{thm:sufficientConditionsForLCwrtTDI} (which also holds without assuming BRS, as can be seen from the proof of the corresponding part of Theorem~\ref{thm:sufficientConditionsForLCwrtTDI}), \eqref{ineq:LipschitzConditionOnf} implies that \eqref{ineq:LipschitzConditionOnf_Feedback} is satisfied for every $\eta \in \KK_\infty \cap \operatorname{Lip}_1$.}
    
    \revblue{\emph{\ref{cond:RFCofCLSystem1}$ \implies $\ref{cond:RFCofCLSystem2}:} F}or every $x \in X$ and $u \in \U_x^\eta$, there exists $d \in \D$ such that $\phi(\ph,x,u) = \phi_\CL(\ph,x,d)$ by \revblue{the argumentation in the proof of \ref{cond:sufficientConditionsForLCwrtTDI4}$\implies$\ref{cond:sufficientConditionsForLCwrtTDI5} of Theorem \ref{thm:sufficientConditionsForLCwrtTDI}.} 
    Hence, \ref{cond:RFCofCLSystem1}$ \implies $\ref{cond:RFCofCLSystem2}.

    \revblue{\emph{\ref{cond:RFCofCLSystem2}$\implies$\ref{cond:RFCofCLSystem3}:} Let Statement \ref{cond:RFCofCLSystem2} hold and \eqref{ineq:LipschitzConditionOnf_Feedback} be satisfied. Then, by the implication \ref{cond:sufficientConditionsForLCwrtTDI3}$\implies$\ref{cond:sufficientConditionsForLCwrtTDI5} of Theorem \ref{thm:sufficientConditionsForLCwrtTDI}, there exists $\widetilde \eta \in \KK_\infty \cap \operatorname{Lip}_1$, such that $\PDESigma$ is RFC with respect to trajectory dominated inputs with margin $\widetilde \eta$ and has a flow that is Lipschitz continuous on compact intervals with respect to trajectory-dominated inputs with the same margin $\widetilde \eta$. Then, Statement \ref{cond:RFCofCLSystem3} follows due to Theorem \ref{thm:BRSLyapunovFunction}.}

    
    \revblue{\emph{\ref{cond:RFCofCLSystem3}$\implies$\ref{cond:RFCofCLSystem1}:} Follows from the implication \ref{cond:sufficientConditionsForLCwrtTDI1}$\implies$\ref{cond:sufficientConditionsForLCwrtTDI4} of Theorem \ref{thm:sufficientConditionsForLCwrtTDI}.}
\end{proof}

\subsection{Relation of notions of Lipschitz continuity of the flow}
We claimed in Remark \ref{rem:relationOfLipschitzContinuityNotions} that Lipschitz continuity on compact intervals with respect to trajectory-dominated inputs does not imply Lipschitz continuity on compact intervals in the general case. We prove this claim with the following example.%

Consider the autonomous scalar system with $X = U = \R$,
\begin{equation}
    \Sigma_1\colon \quad \dot{x} = f(x,u),
\end{equation}
where for any $u \in \R$
\begin{equation*}
    f(x,u) \coloneq
    \begin{cases}
        -\abs{u}x \ln|x| , &\text{if } x\neq 0, \\
        0, &\text{if } x = 0.
    \end{cases}
\end{equation*}
\begin{proposition}\label{prop:LipschitzonCIwrtTDIdoesNotImplyLipschitzonCI}
    $\Sigma_1$ satisfies the following properties:
    \begin{enumerate}
        \item\label{cond:LipschitzonCIwrtTDIdoesNotImplyLipschitzonCI1} $f$ is not locally Lipschitz continuous at $0$ in the first argument \revblue{for fixed $u \neq 0$}.
        \item\label{cond:LipschitzonCIwrtTDIdoesNotImplyLipschitzonCI2} The solutions of $\Sigma_1$ exist globally and are unique.
        \item\label{cond:LipschitzonCIwrtTDIdoesNotImplyLipschitzonCI3} The flow of $\Sigma_1$ is not Lipschitz continuous on compact intervals.
        \item\label{cond:LipschitzonCIwrtTDIdoesNotImplyLipschitzonCI4} Statement \ref{cond:sufficientConditionsForLCwrtTDI3} of Theorem \ref{thm:sufficientConditionsForLCwrtTDI} holds. In particular, the flow of $\Sigma_1$ is Lipschitz continuous on compact intervals with respect to trajectory-dominated inputs.        
    \end{enumerate}
\end{proposition}
    
\begin{proof}
    \emph{Claim \ref{cond:LipschitzonCIwrtTDIdoesNotImplyLipschitzonCI1}:} Clear, as the derivative of the right\revblue{-}hand side is unbounded when $x\to 0$.


\emph{Claim \ref{cond:LipschitzonCIwrtTDIdoesNotImplyLipschitzonCI2}:}
    Let $C > 0$ and $u \in B_{C,\U}$. At first, we show that local solutions exist which are unique: If $x(0)\neq 0$, the claim is clear due to Picard-Lindelöf theorem since $f$ is locally Lipschitz continuous for $x \neq 0$. Let $x(0)=0$.
    Then, the map
    \[
    \omega(r) \coloneq 
    \begin{cases}
        \revblue{0}, &\revblue{\text{if } r = 0,} \\
        - C \cdot r\ln\argparen{\revblue{\tfrac{r}{2}}}, &\text{if } r \revblue{> 0},
    \end{cases}
    \]
    \revblue{is a local} modulus of continuity for $f$ \revblue{on $[-e^{-1},e^{-1}]$}\revblue{, i.e., 
    \begin{align*}
        \abs{f(x,u)-f(y,u)} \leq \omega\argparen{\abs{x - y}}\quad x,y \in [-e^{-1},e^{-1}].
    \end{align*}
    }
    Moreover, $\omega$ is continuous and nondecreasing. Since
    \[
    \int_0^{e^{-1}} \frac{1}{\omega(r)} \diff r
    =
    \int_0^{e^{-1}} \frac{1}{-C \cdot r\ln\argparen{\revblue{\tfrac{r}{2}}}} \diff r
    = \infty
    \]
    holds, the Osgood uniqueness criterion (see \cite[§ 12.VIII, p. 147]{Wal98}) applies.
    Therefore, solutions of $\Sigma_1$ are unique even at the non-Lipschitz equilibrium.
    
    Moreover, $\Sigma_1$ is BRS: It holds that $f(x,u) \leq 0$ for $x \geq 1$ and $f(x,u) \geq 0$ for $x \leq -1$, i.e., for every solution $\phi = \phi(t,x,u)$, the inequality $\abs{\phi(t,x,u)} \leq \max\{1,\abs{x}\}$ is satisfied.  Then, for each initial condition and input, the (maximal) \revblue{solution exists} for all $t \in \R_+$ (see Section \ref{sec:LyapunovODEs}).
    
    \emph{Claim \ref{cond:LipschitzonCIwrtTDIdoesNotImplyLipschitzonCI3}:} Let $u \equiv 1$. A straightforward computation shows that the trajectories 
    are given by
    \begin{align*}
        \phi(t,x,1) = \sign(x)\abs{x}^{e^{-t}},\quad x \in\R.
    \end{align*}
    Hence, for 
    $t = \ln 3$, we have
    \begin{align*}
        \phi(\ln 3,x,1) = \sqrt[3]{x},\quad x \in\R.
    \end{align*}
    Obviously, $\sqrt[3]{x}$ is not Lipschitz continuous in $0$ and therefore the flow of $\Sigma_1$ is not Lipschitz continuous on compact intervals.
    
    \emph{Claim \ref{cond:LipschitzonCIwrtTDIdoesNotImplyLipschitzonCI4}:}
    Let $\eta \in \KK_\infty$ be given by
    \begin{align*}
        \eta(s) \coloneq 
        \begin{cases}
            \revblue{0,} &\revblue{\text{if } s = 0,} \\
            -\frac{s}{2\ln(s)}, &\text{if } s \in \revblue{(}0,e^{-1}],\\
            \frac{1}{2}s, &\text{if } s > e^{-1}.
        \end{cases}
    \end{align*}
    \revblue{Then, $\eta$ is differentiable for $s \neq e^{-1}$ with}
    \begin{align*}
        \frac{\diff}{\diff s}\eta(s) &=
        \begin{cases}
            \frac{1}{2\ln(s)}\paren{\frac{1}{\ln(s)} - 1}, &\text{if } s \in \revblue{(0,e^{-1})},\\
            \frac{1}{2}, &\text{if } s > e^{-1},
        \end{cases}
        \\
        &\leq 1,
    \end{align*}
    \revblue{i.e., $\eta \in \operatorname{Lip}_1$.}
    For $d \in [-1,1]$, we have
    \begin{align}
        f(x,d\eta\argparen{\abs{x}})
        &= 
        \begin{cases}
            0, &\text{if } \abs{x} = 0, \\
            -\abs{-d\frac{\abs{x}}{2\ln\abs{x}}}x \ln|x|, &\text{if } \abs{x} \in (0,e^{-1}],\\
            -\tfrac{1}{2}\abs{dx}x \ln|x|,\! &\text{if } \abs{x} > e^{-1},
        \end{cases}
        \nonumber\\
        &= 
        \begin{cases}
            \tfrac{1}{2}\abs{d}x\abs{x}, &\text{if } \abs{x} \in [0,e^{-1}],\\
            -\tfrac{1}{2}\abs{d} x\abs{x} \ln|x|,\! &\text{if } \abs{x} > e^{-1}.
        \end{cases}
        \label{eq:fOfFeedbackSystem}
    \end{align}

    We note that $f(x,d\eta\argparen{\abs{x}}) \leq 0$ for $x \geq 1$ and $f(x,d\eta\argparen{\abs{x}}) \geq 0$ for $x \leq -1$, i.e., for every solution $\overline \phi = \overline \phi(t,x,d)$, the inequality $\abs{\overline \phi(t,x,d)} \leq \max\{1,\abs{x}\}$ is satisfied. In particular, $\Sigma_1$ is \revblue{RFC} with respect to trajectory-dominated inputs for growth margin $\eta$.
    
    Then, by \eqref{eq:fOfFeedbackSystem}, the map $\R \to \R$, $x \mapsto f\argparen{x,d\eta\argparen{\abs{x}}}$, $d \in[-1,1]$, is differentiable \revblue{for $\abs{x} \neq e^{-1}$. If $\abs{x} \in [0,e^{-1})$, then}
    \begin{align*}
        \abs{\tfrac{\diff}{\diff x}f(x,d\eta\argparen{\abs{x}})}
        &= \tfrac{1}{2}\abs{d}\abs{x} + \tfrac{1}{2}\abs{d}x \tfrac{\abs{x}}{x}
        = \abs{d}\abs{x}
        \leq \abs{x}
    \end{align*}
    \revblue{On the other hand, if $\abs{x} > e^{-1}$, then}
    \begin{align*}
        &\abs{\tfrac{\diff}{\diff x}f(x,d\eta\argparen{\abs{x}})} \nonumber\\
        &\qquad=\tfrac{1}{2}\abs{-\abs{d} \abs{x} \ln|x| -\abs{d} x\tfrac{\abs{x}}{x} \ln|x| - \abs{d} x\abs{x} \tfrac{1}{x}} \nonumber\\
        &\qquad= \tfrac{1}{2}\abs{d}\abs{x} \abs{2\ln|x| + 1} 
        \leq \tfrac{1}{2}\abs{x} \max\argbraces{1,\abs{2x}} \\
        &\qquad\leq \max\argbraces{\abs{x}, x^2},
    \end{align*}
    where we used $\ln(s)\leq s - 1$ for all $s > 0$ and $\ln(s) > -1$ for $s > e^{-1}$ in the third step. Then, it holds that
    \begin{align}
        \abs{\tfrac{\diff}{\diff x}f(x,d\eta\argparen{\abs{x}})}
        &\leq \max\argbraces{\abs{x}, x^2}. \label{ineq:estimateSigma1LipschitzBound}
    \end{align}
    
    Now, pick any $C \geq 1$ and any $x_1,x_2 \in B_C$.
    By the mean value theorem applied to \revblue{the intervals without junctions of} $x\mapsto f(x,d\eta(|x|))$ and \eqref{ineq:estimateSigma1LipschitzBound}, we obtain
    \begin{align*}
        \abs{f(x_1,d\eta\argparen{\abs{x_1}}) - f(x_2,d\eta\argparen{\abs{x_2}})}
        &\leq C^2 \abs{x_1 - x_2}.
    \end{align*}
    
    As $\eta \in \operatorname{Lip}_1$ and \revblue{RFC} with respect to trajectory-dominated inputs for growth margin $\eta$, we see that Statement \ref{cond:sufficientConditionsForLCwrtTDI3} of Theorem \ref{thm:sufficientConditionsForLCwrtTDI} holds.

    In view of BRS of $\Sigma_1$, the application of Theorem \ref{thm:sufficientConditionsForLCwrtTDI} shows that $\Sigma_1$ is Lipschitz continuous on compact intervals with respect to trajectory-dominated inputs.
\end{proof}

\section{Lyapunov characterization of forward completeness for ODE systems}\label{sec:LyapunovODEs}

    Having proved a Lyapunov criterion for BRS of abstract infinite-dimensional control systems, we would like to prove a novel Lyapunov characterization of forward completeness for ODE systems with inputs of arbitrary magnitude.
    
Let $\Sigma$ be an ODE system given by
\begin{align*}
    \finiteSigma\colon \dot x(t) = f(x(t),u(t)), \quad t \in \I,
\end{align*}
where 
$x \in X = \R^n$, $U = \R^m$, for some natural numbers $m,n$, $u \in \Uc \coloneq \revblue{\LL}^\infty(\R_+,\R^m)$ and 
\revblue{$f$ satisfies the Lipschitz condition \eqref{ineq:LipschitzConditionOnf}}. 
Clearly, such ODEs are a subclass of \revblue{semilinear evolution equations} $\PDESigma$.

Lipschitz continuity of $f$ on bounded sets ensures that for any initial condition $x \in X$ and any input $u\in\Uc$, the corresponding maximal solution (in the sense of Carathéodory) $\phi(\ph,x,u)$ of $\finiteSigma$ exists and is unique on a certain \revblue{nonempty} interval. 
Furthermore, $\Sigma\coloneq(X,\Uc,\phi)$ is a well-defined control system with the BIC property; see \cite[Theorem 1.16, Proposition~1.20]{Mir23}. 

\begin{theorem}\label{thm:FCLyapunovFunction}
    Let $\finiteSigma$ be an ODE system and     \revblue{\eqref{ineq:LipschitzConditionOnf}}
    be satisfied.
	The following statements are equivalent:
    \revblue{
    \begin{enumerate}
        \item\label{cond:FCLyapunovFunction1} $\finiteSigma$ is FC.
        \item\label{cond:FCLyapunovFunction2} $\finiteSigma$ is BRS.
        \item\label{cond:FCLyapunovFunction3} There is a continuous BRS Lyapunov function for $\finiteSigma$.
        \item\label{cond:FCLyapunovFunction4} There exists a BRS Lyapunov function for $\finiteSigma$ that is \revblue{Lipschitz continuous on bounded sets}.        
        \item\label{cond:FCLyapunovFunction5} $\finiteSigma$ is \revblue{RFC} w.r.t. trajectory-dominated inputs.
    \end{enumerate}
    }
\end{theorem}

\begin{proof}
\revblue{(v) implies (i) by definition.
In \cite[Proposition 5.1]{LSW96}, the equivalence between (i) and (ii) was proved. 
Let (ii) hold. As $f$ is Lipschitz continuous on bounded balls, Statement \ref{cond:sufficientConditionsForLCwrtTDI1} of Theorem~\ref{thm:sufficientConditionsForLCwrtTDI} holds, and   
Theorem~\ref{thm:sufficientConditionsForLCwrtTDI} infers that also statement (v) of this theorem holds, which together with the BIC property of $\finiteSigma$ allows us to invoke} our converse Lyapunov Theorem~\ref{thm:BRSLyapunovFunction} to infer that (ii) $\Rightarrow$ (iv)  $\Rightarrow$ (iii) $\Rightarrow$ (ii) $\Rightarrow$ (v).    
\end{proof}

Note that the equivalences between items \ref{cond:FCLyapunovFunction2} -- \ref{cond:FCLyapunovFunction5} can be obtained under somewhat weaker assumptions on the system $\finiteSigma$. However, the available results ensuring the equivalence between items \ref{cond:FCLyapunovFunction1} and \ref{cond:FCLyapunovFunction2} require continuity of $f$ together with Lipschitz continuity of $f$ with respect to the first variable; see \cite[Theorem 1.31]{Mir23}.

Theorem~\ref{thm:FCLyapunovFunction} is new and differs from \cite[Theorem 2]{AnS99}, where Lyapunov characterization of forward completeness for systems with inputs restricted to the set $\D\coloneq\{u \in\Uc: \|u\|_\Uc \leq R\}$ with a given finite $R$ has been shown. An infinite-dimensional extension of this result from \cite{AnS99} has been obtained in \cite{Mir23e}, where \revblue{RFC} has been characterized in terms of Lyapunov functions.

\section{Undisturbed systems}\label{sec:BRSUndisturbed}
We \revblue{specialize} Theorem \ref{thm:BRSLyapunovFunction} \revblue{to} \emph{undisturbed systems}, i.e., systems which do not depend on $u$ such that the trajectories reduce to $\phi(t,x,u) = \phi(t,x)$ for all $t \in \R_+$, $x \in X$ and $u \in \U$. Then, Lipschitz continuity on compact intervals does not depend on the inputs, and \revblue{for any nonempty $\D \subset \U$,} BRS is equivalent to \revblue{RFC both} with respect to inputs in $\D$ and with respect to trajectory-dominated inputs.

Note that $V \in \revblue{\Cont}(X,\R_+)$ is a BRS Lyapunov function for an undisturbed system $\Sigma$, if the following two conditions hold:
\begin{enumerate}
    \item\label{def:undistBRSLyapunovFunctionCondition1} There exist $\alpha_1, \alpha_2 \in \KK_\infty$ and $C \geq 0$ such that
    \begin{align*}
        \alpha_1\!\paren{\norm{x}_X} &\leq V(x) \leq \alpha_2\!\paren{\norm{x}_X} + C, \quad x \in X.
    \end{align*}
    \item\label{def:undistBRSLyapunovFunctionCondition2*} For all $x \in X$, it holds that
    \begin{align}
    \label{eq:RFC-estimate}
        \dot V(x) \leq V\!\paren{x}.
    \end{align}
\end{enumerate}

\begin{remark}
    In \cite{Mir23e}, instead of \eqref{eq:RFC-estimate} the estimate
\[
\exists M>0:\quad \dot V(x) \leq V\!\paren{x} + M\quad x\in X
\]
was employed. 
However, setting $W(x)\coloneq V(x)+M$ leads to $\dot W(x) \leq W(x)$ and $W(x) \leq \alpha_2\!\paren{\norm{x}_X} + C+M$, thus our formulation here and that in \cite{Mir23e} are equivalent.
\end{remark}

For undisturbed systems, our characterization of BRS takes the form
\begin{corollary}
\label{cor:UndisturbedLyapunovFunction}
	Let $\Sigma$ be an undisturbed system \revblue{satisfying the BIC property}, 
    and the flow of $\Sigma$ be Lipschitz continuous on compact intervals. 
    \revblue{Fix a nonempty $\D \subset \U$.} 
    The following statements are equivalent:
    \begin{enumerate}
        \item\label{cond:UndisturbedLyapunovFunction1} $\Sigma$ is BRS.
        \item\label{cond:UndisturbedLyapunovFunction2} $\Sigma$ is \revblue{RFC} with respect to inputs in $\D$.
        \item\label{cond:UndisturbedLyapunovFunction3} There is a (continuous) BRS Lyapunov function for $\Sigma$.
        \item\label{cond:UndisturbedLyapunovFunction4} There exists a BRS Lyapunov function for $\Sigma$ that is \revblue{Lipschitz continuous on bounded sets}.        
        \item\label{cond:UndisturbedLyapunovFunction5} $\Sigma$ is \revblue{RFC} with respect to trajectory-dominated inputs.
    \end{enumerate}
\end{corollary}
    
\revblue{It can still be advantageous} to extend an undisturbed system to a control system by an appropriate choice for a \revblue{fictitious input}. To illustrate this phenomenon, we consider the following example.
\begin{example}\label{ex:undisturbedSystem}
    Consider the autonomous scalar system with $X = U = \R$,
    \begin{equation*}
        \Sigma_2\colon \quad \dot{x} = f(x), \qquad         f(x) \coloneq
        \begin{cases}
            -x \ln|x| , &\text{if } x\neq 0, \\
            0, &\text{if } x = 0.
        \end{cases}
    \end{equation*}
    \revblue{The flow of $\Sigma_2$ equals $x\mapsto\phi(t,x,1)$ for the controlled system $\Sigma_1$; hence Proposition~\ref{prop:LipschitzonCIwrtTDIdoesNotImplyLipschitzonCI} shows that it}
    is not Lipschitz continuous  on compact intervals. Therefore, Corollary \ref{cor:UndisturbedLyapunovFunction} cannot be used to conclude \revblue{the} existence of a BRS Lyapunov function. 
    However, in Proposition \ref{prop:LipschitzonCIwrtTDIdoesNotImplyLipschitzonCI} we showed that the flow of $\Sigma_1$ is Lipschitz continuous on compact intervals with respect to trajectory-dominated inputs. 
    
    That is, a BRS Lyapunov function for $\Sigma_1$ exists. As a consequence, we obtain a BRS Lyapunov-like function for $\Sigma_2$ by setting $u \equiv 1$.
\end{example}

However, Lyapunov functions constructed via this approach of extending the system to a control system have the form
\begin{enumerate}
    \item\label{def:varUndistBRSLyapunovFunctionCondition1} There exist $\alpha_1, \alpha_2 \in \KK_\infty$ and $C \geq 0$ such that
    \begin{align*}
        \alpha_1\!\paren{\norm{x}_X} &\leq V(x) \leq \alpha_2\!\paren{\norm{x}_X} + C, \quad x \in X.
    \end{align*}
    \item\label{def:varUndistBRSLyapunovFunctionCondition2*} There exists $\eps > 0$ such that \revblue{for all $x \in X$, it holds that}
    \begin{align*}			
        \eps \leq
        \norm x_X 
        \qquad\implies
        \qquad \dot V(x) \leq V\!\paren{x}.
    \end{align*}
\end{enumerate}
Hence, the difference between the two approaches is that a neighborhood of the problematic point $x = 0$ is excluded in the latter.

If a control system \revblue{satisfies BIC property}, then the existence of a Lyapunov function as above will prove BRS, as for this property, we are not interested in what happens within any given ball of a finite radius, and we are rather interested in what happens outside of this ball.

\section{Applications and Examples}\label{sec:applicationsAndExamples}
In the following, we employ BRS Lyapunov functions to develop criteria for BRS for two practical scenarios.

\subsection{Example: Wintner's theorem for ODEs}

Based on our definition of a BRS Lyapunov function, we can give a Lyapunov-based proof of the well-known Wintner-type criterion for forward completeness and BRS with an explicit construction of a BRS Lyapunov function.
\begin{proposition}[Wintner-type criterion for BRS]
Consider the ODE system
\begin{equation}
    \dot{x}(t)=f(x(t)), \qquad t\in\R_+,
    \label{eq:Wintner-ODE}
\end{equation}
in $\R^n$, $n \in \N$ equipped with norm $\norm{\ph}$,
where $f:\R^n\to\R^n$ is locally Lipschitz continuous.
Assume that there exists a continuous function
$L:\R_+\to(0,\infty)$ such that
\begin{equation}
    \norm{f(x)}\leq L(\norm{x}),
    \qquad x\in\R^n,
    \label{eq:Wintner-growth}
\end{equation}
and
\begin{equation}
    \int_c^\infty \frac{1}{L(s)}\,\diff s=+\infty
    \qquad \text{for every } c>0.
    \label{eq:Wintner-integral}
\end{equation}
Then, \eqref{eq:Wintner-ODE} is forward complete and has BRS.
\end{proposition}

\begin{proof}
Let us define
\begin{equation*}
    \rho(r)\coloneq\int_0^r\frac{1}{L(s)}\,\diff s,
    \qquad r\in\R_+.
\end{equation*}
Since $L$ is continuous and positive, $\rho$ is continuously
differentiable and strictly increasing. Moreover,
\eqref{eq:Wintner-integral} implies that $\rho$ is unbounded.
Hence, $\rho\in\Kinf$.

Consider first the function $W\colon \R^n \to \R_+$
\begin{align}\label{eq:WintnerSandwichBound}
    W(x)\coloneq\rho(\norm{x}), \qquad x\in\R^n.
\end{align}
We estimate its Lie derivative along the solutions of
\eqref{eq:Wintner-ODE}. By the definition of the Lie derivative,
\[
    \dot W(x)
    =
    \limsup\nolimits_{t \searrow 0}
    \tfrac{
        \rho(\norm{\phi(t,x)})-\rho(\norm{x})
    }{t}.
\]
Since
\[
    \phi(t,x) \in x+t f(x)+o(t), \qquad t\searrow0,
\]
where $o$ is the Landau notation, the reverse triangle inequality yields $\norm{\phi(t,x)}-\norm{x} \leq t\norm{f(x)}+o(t)$.
Using that $\rho'(r)=\frac{1}{L(r)}$ and the mean value theorem, we
therefore obtain
\begin{align}\label{ineq:WintnerDissipation}
    \dot W(x)
    &\leq
    \rho'(\norm{x})\norm{f(x)}=
    \tfrac{\norm{f(x)}}{L(\norm{x})}
    \leq 1,
\end{align}
where the last inequality follows from
\eqref{eq:Wintner-growth}.

Then, \eqref{eq:WintnerSandwichBound} and \eqref{ineq:WintnerDissipation} define a BRS Lyapunov function of the type introduced in Corollary \ref{cor:variationBRSLyapunovFunction}.


Since $f$ is locally Lipschitz continuous, the corresponding ODE
has the BIC property. Hence, the existence of the BRS Lyapunov
function $W$ implies that \eqref{eq:Wintner-ODE} has BRS and, in
particular, is forward complete.
\end{proof}

\begin{remark}
    A BRS Lyapunov function as given in Definition \ref{def:BRSLyapunovFunction} can be constructed by the map $V\colon X \to \R_+$, $V(x) = e^{W(x)}$; also see the construction in Corollary \ref{cor:variationBRSLyapunovFunction}.
\end{remark}


\subsection{Example: Linear boundedness}
We want to employ our results to prove BRS for a class of semilinear evolution equations. We remark that an alternative to the present approach is to use Gronwall-Bellman inequality for a direct estimate. 





\begin{proposition}
Let the semilinear evolution equation
\begin{equation}
    \PDESigma\colon
    \dot x(t)=Ax(t)+f(x(t),u(t)), \qquad t\in\R_+
\end{equation}
be as in \eqref{eq:EE-system}, where $A:D(A)\to X$ is the infinitesimal generator of a
$\Cont_0$-semigroup $(T(t))_{t\geq 0}$ on $X$, and assume that
$\PDESigma$ has the BIC property.

Assume that there exist $a,b>0$ and $\gamma\in\Kinf$ such that
\begin{align}
    \|f(x,u)\|_X
    &\leq
    \big(a+\gamma(\|u\|_U)\big)\|x\|_X
    +b+\gamma(\|u\|_U), \nonumber\\
    &\qquad\qquad\qquad\qquad\qquad\quad x\in X,\ u\in U.
    \label{eq:Lin-boundedness}
\end{align}
Then $\PDESigma$ has BRS.
\end{proposition}

\begin{proof}
Since $(T(t))_{t\geq 0}$ is a $\Cont_0$-semigroup, there exist
$M\geq 1$ and $\omega\in\R$ such that
\[
    \|T(t)\|\leq M e^{\omega t},
    \qquad t\geq 0.
\]
Fix $\lambda>\max\{\omega,0\}$ and introduce the equivalent norm
\[
    \|x\|_\lambda
    \coloneq
    \sup_{t\geq 0}e^{-\lambda t}\|T(t)x\|_X,
    \qquad x\in X.
\]
Then
\begin{equation}
    \|x\|_X
    \leq
    \|x\|_\lambda
    \leq
    M\|x\|_X,
    \qquad x\in X,
    \label{eq:eq-norm-semigroup}
\end{equation}
and
\begin{equation}
    \|T(t)x\|_\lambda
    \leq
    e^{\lambda t}\|x\|_\lambda,
    \qquad t\geq 0,\ x\in X.
    \label{eq:quasi-contractivity}
\end{equation}

We construct a BRS Lyapunov function for $\PDESigma$.
Let $x\in X$, $u\in\mathcal U$.
By the definition of a mild solution \eqref{eq:mildSolution}, \eqref{eq:quasi-contractivity},
\eqref{eq:eq-norm-semigroup}, and
\eqref{eq:Lin-boundedness}, for $t \in [0,t_{\max}(x,u))$ we obtain
\begin{align*}
    &\|\phi(t,x,u)\|_\lambda \\
    &\leq
    e^{\lambda t}\|x\|_\lambda
    +
    M\int_0^t e^{\lambda(t-s)}
    \|f(\phi(s,x,u),u(s))\|_X\diff s
    \\
    &\leq
    e^{\lambda t}\norm{x}_\lambda \\
    &+
    M\!\!\!\int_0^t\!\!\! e^{\lambda(t-s)}
    \Big(\!
        \big(a\!+\!\gamma(\norm{u}_\U)\big)
        \|\phi(s,x,u)\|_\lambda\!
        +\!b\!+\!\gamma(\norm{u}_\U)\!
    \Big)\!\diff s .
\end{align*}
Since $s\mapsto\phi(s,x,u)$ is continuous, taking the Lie derivative at $t=0$ yields
%
%
%
\begin{align}
    &\tfrac{\diff}{\diff t} \paren{\norm{\ph}_\lambda}_u(x)
    = \limsup_{t \searrow 0} \tfrac{\norm{\phi(t,x,u)}_\lambda - \norm{x}_\lambda}{t} \nonumber\\
    &\leq
    \lambda \norm{x}_\lambda \!
    +
    M\big(a+\gamma(\norm{u}_\U)\big)\norm{x}_\lambda \!
    +
    M\big(b+\gamma(\norm{u}_\U)\big).
    \label{eq:Dini-equivalent-norm}
\end{align}


We set $c\coloneq\lambda+M(a+b)$ and define $V\colon X \to \R_+$,
\begin{equation}
    V(x)\coloneq
    c+\ln\big(1+\|x\|_\lambda\big),
    \qquad x\in X.
    \label{eq:linear-growth-BRS-Lyap}
\end{equation}
By \eqref{eq:Dini-equivalent-norm} and the chain rule for the
Lie derivative,
\begin{align}
    \dot V_u(x)
    &\leq
    \tfrac{1}{1+\norm{x}_\lambda}\Big(\lambda \norm{x}_\lambda
        +M\big(a+\gamma(\norm{u}_\U)\big)\norm{x}_\lambda \nonumber\\
        &\qquad+M\big(b+\gamma(\norm{u}_\U)\big)\!\Big) 
    \nonumber\\
    &\leq
    c+M\gamma(\norm{u}_\U).
    \label{eq:W-Dini-estimate}
\end{align}
Furthermore, we define $\chi\in\Kinf$,
\begin{equation}
    \chi(s)\coloneq e^{M\gamma(s)}-1,
    \qquad s\in\R_+.
    \label{eq:linear-growth-chi}
\end{equation}

Suppose now that $\chi(\|u\|_{\mathcal U})\leq \|x\|_X$.
By \eqref{eq:eq-norm-semigroup}, we have
\[
    e^{M\gamma(\norm{u}_\U)}-1
    =
    \chi(\norm{u}_\U)
    \leq
    \|x\|_X
    \leq
    \|x\|_\lambda,
\]
and hence
\[
    M\gamma(\norm{u}_\U)
    \leq
    \ln\big(1+\|x\|_\lambda\big).
\]
Together with \eqref{eq:W-Dini-estimate}, this gives
\[
    \dot V_u(x)
    \leq
    c+M\gamma(\norm{u}_\U)
    \leq
    c+\ln\big(1+\|x\|_\lambda\big)
    =
    V(x).
\]
Thus, the growth condition in Definition~\ref{def:BRSLyapunovFunction} is satisfied.

It remains to verify the coercivity bounds. By
\eqref{eq:eq-norm-semigroup},
\[
    \ln(1+\|x\|_X)
    \leq
    V(x)
    \leq
    \ln(1+M\|x\|_X)+c.
\]
Hence, with $\alpha_1,\alpha_2\in\Kinf$, $\alpha_1(s)\coloneq\ln(1+s)$, $\alpha_2(s)\coloneq\ln(1+Ms)$, we have
\[
    \alpha_1(\|x\|_X)
    \leq
    V(x)
    \leq
    \alpha_2(\|x\|_X)+c,
    \qquad x\in X.
\]
Consequently, $V$ is a BRS Lyapunov function for
$\PDESigma$. 
\revblue{By \cite[Proposition 4.4]{Mir23e}, the BIC property and the existence of a BRS Lyapunov function imply BRS.}
Therefore,
$\PDESigma$ has BRS.
\end{proof}
\color{black}

\section{Conclusion}\label{sec:conclusion}
In this article, we presented a notion of BRS Lyapunov function and developed regularity conditions such that the existence of a BRS Lyapunov function is equivalent to BRS. Moreover, we provided several criteria for BRS. Here, especially \revblue{RFC} with respect to trajectory-dominated inputs should be mentioned.

\revblue{As a direction for future work,} Lyapunov characterizations for certain non-forward complete infinite-dimensional systems with outputs should be developed. In the special case of finite-dimensional ODE systems, the Lyapunov characterization of forward completeness (which is equivalent to BRS for such systems) was essential \cite{AnS99} to achieve this goal. However, in the infinite-dimensional case, the concept of unboundedness observability \revblue{must be appropriately extended to infinite-dimensional systems.}

Another interesting research direction is to use our strategy to give a direct proof of a Lyapunov characterization of ISS for infinite-dimensional systems. The covered system class may be extended to include abstract control systems, as in this paper, as opposed to the existing approach \cite{MiW17c}, developed for evolution equations with Lipschitz nonlinearities.

\section*{References}
\bibliographystyle{IEEEtran}
\bibliography{Bibliothek,MyPublications}

\begin{IEEEbiography}[{\includegraphics[width=1in,height=1.25in,clip,keepaspectratio]{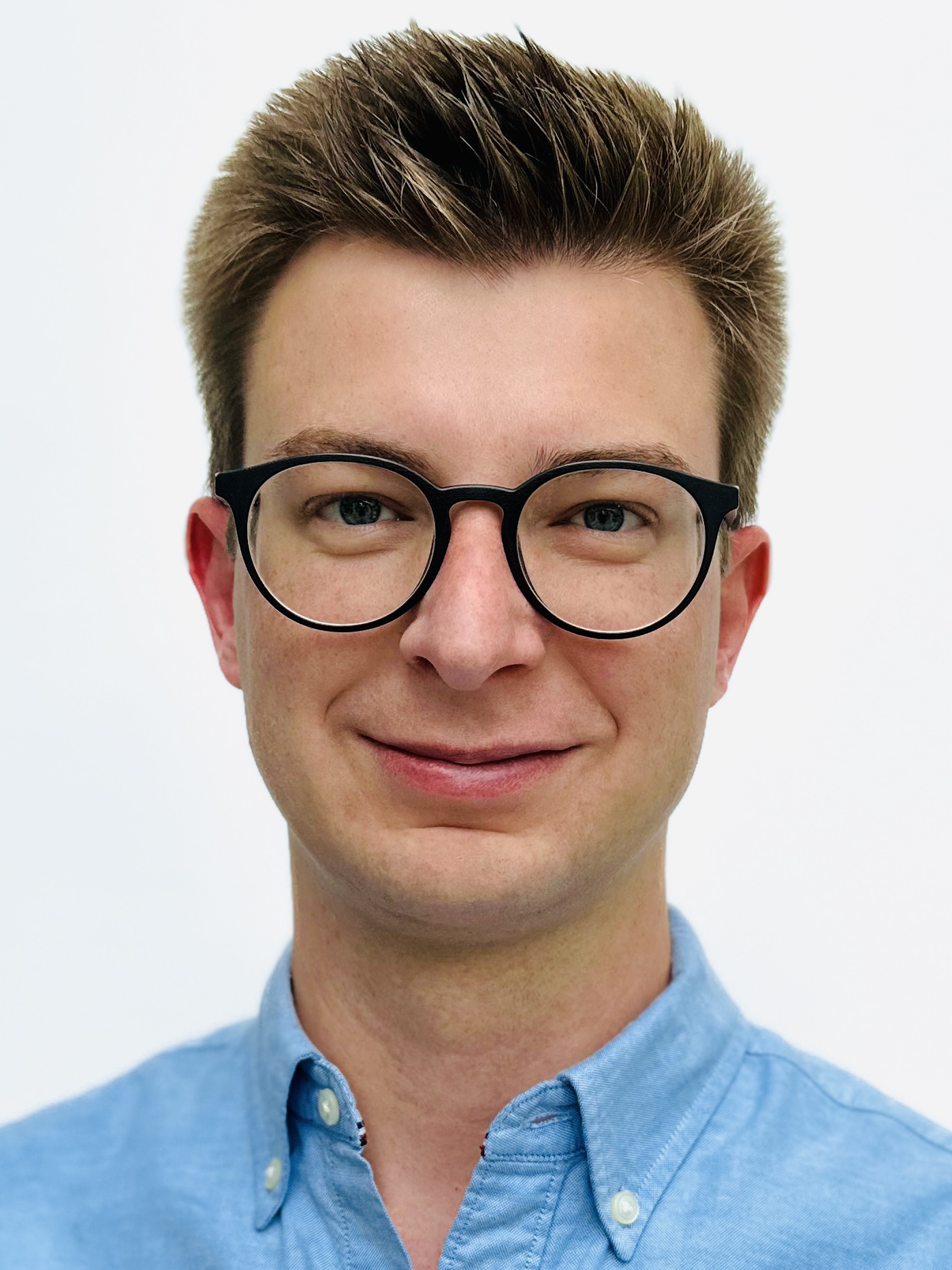}}]{Patrick Bachmann} received his Bachelor's degree in business mathematics from the University of Mannheim, Germany, in 2015 and his Master's degree in Mathematics from Karlsruhe Institute of Technology, Germany, in 2018. He worked as a research assistant at the Technical University of Kaiserslautern, Germany, and the University of Würzburg, Germany, and the University of Bayreuth, Germany. Currently, he is pursuing his PhD degree in Mathematics under the supervision of Sergey Dashkovskiy and Andrii Mironchenko at University of Würzburg. His research interests include impulsive systems, stability and control theory, Lyapunov functions, and infinite-dimensional systems.
\end{IEEEbiography}

\begin{IEEEbiography}[{\includegraphics[width=1in,height=1.25in,clip,keepaspectratio]{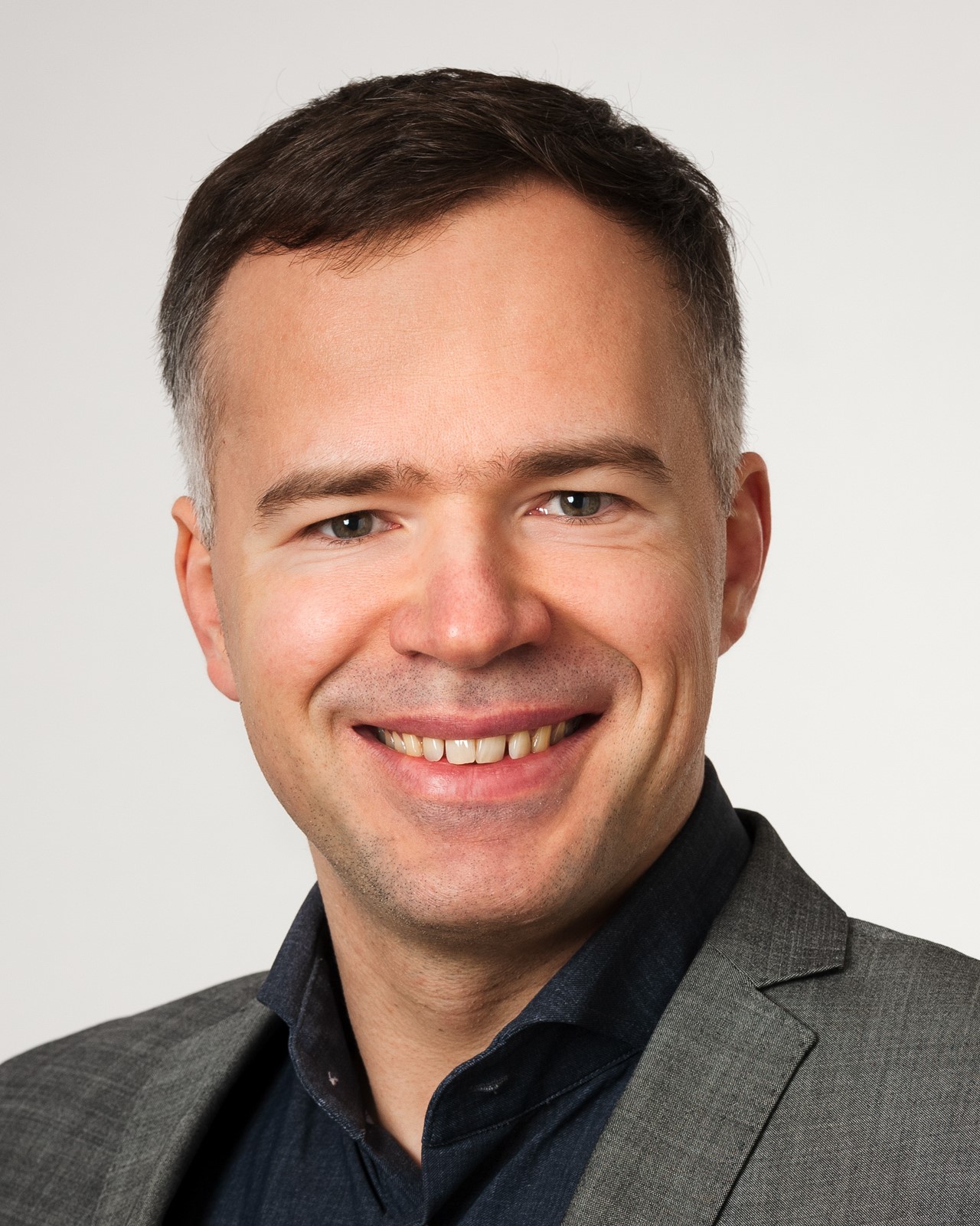}}]{Andrii Mironchenko} (M'21-SM'22) 
was born in 1986 in Odesa, Ukraine. He received his Ph.D. degree in mathematics from the University of Bremen, Germany (2012), and a habilitation degree from the University of Passau, Germany (2023). He was a Postdoctoral Fellow of the Japan Society for Promotion of Science (2013–2014). Since December 2024, he has been with the Department of Mathematics, University of Bayreuth, Germany.

Dr. Mironchenko is the author of the monograph “Input-to-State Stability: Theory and Applications” (Springer, 2023). He is an Associate Editor in Systems \& Control Letters (2023 --) and IEEE Transactions on Automatic Control (2026 --). A. Mironchenko is a co-founder and co-organizer of the biennial Workshop series “Stability and Control of Infinite-Dimensional Systems” (SCINDIS, 2016 --) and ISS Online Seminar (2021 --). He is a recipient of IEEE CSS George S. Axelby Outstanding Paper Award (2023), Outstanding Habilitation Award of the University of Passau (2024), Heisenberg grant (2024), and von Kaven Award (2025) from the German Research Foundation.

His research interests include stability theory, nonlinear systems theory, distributed parameter systems, hybrid systems, and applications of control theory to biological systems and distributed control. 
\end{IEEEbiography}

\end{document}